\documentclass[a4paper, 11pt]{amsart}
\usepackage[utf8]{inputenc}
\usepackage[T1]{fontenc}
\usepackage[english]{babel}
\usepackage{graphicx}
\usepackage{stmaryrd}
\usepackage{tikz}
\usepackage{tikz-cd} 
\usepackage[all]{xy}
\usepackage{comment}
\usepackage{amsmath,amsfonts,amssymb}
\usepackage[a4paper]{geometry}
\usepackage{colortbl}
\usepackage{amsthm}
\usepackage{float}
\usepackage{enumitem}
\usepackage{hyperref}
\usepackage{xcolor}
\hypersetup{
    colorlinks=true,
    linkcolor={red},
    citecolor={blue},
    urlcolor={blue}}
\newtheorem{te}{Theorem}[section]

\newtheorem{prop}[te]{Proposition}
\newtheorem{co}[te]{Corollary}
\newtheorem{lemme}[te]{Lemma}
\theoremstyle{definition}
\newtheorem{de}[te]{Definition}

\theoremstyle{remark}
\newtheorem{rque}[te]{Remark}
\newenvironment{proof2}{\noindent\textit{Proof of Lemma \ref{penible}.~}}{\hfill$\square$\bigbreak} 
\newlength{\plarg}
\usetikzlibrary{cd}
\calclayout
\title{Virtually free groups are almost homogeneous}
\author{Simon André}
\date{\today}
\begin{document}
\clearpage
\begin{minipage}{\linewidth}
\begin{abstract}Free groups are known to be homogeneous, meaning that finite tuples of elements which satisfy the same first-order properties are in the same orbit under the action of the automorphism group. We show that virtually free groups have a slightly weaker property, which we call uniform almost-homogeneity: the set of $k$-tuples which satisfy the same first-order properties as a given $k$-tuple $\mathbf{u}$ is the union of a finite number of $\mathrm{Aut}(G)$-orbits, and this number is bounded independently from $\mathbf{u}$ and $k$. Moreover, we prove that there exists a virtually free group which is not $\exists$-homogeneous. We also prove that all hyperbolic groups are homogeneous in a probabilistic sense.
\end{abstract}
\maketitle
\end{minipage}

\section{Introduction}
Let $G$ be a group and let $k\geq 1$ be an integer. We say that two tuples $\mathbf{u}=(u_1,\ldots ,u_k)$ and $\mathbf{v}=(v_1,\ldots ,v_k)$ in $G^k$ have the same \textit{type} if, given any first-order formula $\theta(x_1,\ldots ,x_k)$ with $k$ free variables, the statement $\theta(\mathbf{u})$ is true in $G$ if and only if the statement $\theta(\mathbf{v})$ is true in $G$. We use the notation $\mathrm{tp}(\mathbf{u})=\mathrm{tp}(\mathbf{v})$. Roughly speaking, two tuples have the same type if they are indistinguishable from the point of view of first-order logic. Obviously, two tuples that are in the same orbit under the action of the automorphism group of $G$ have the same type. The group $G$ is termed \textit{homogeneous} if the converse holds, i.e.\ if the power of expression of first-order logic is strong enough to distinguish between two tuples that belong to different orbits under the action of the automorphism group of $G$. In other words, $G$ is homogeneous if two finite tuples are indistinguishable from the point of view of logic if and only if they are indistinguishable from the point of view of algebra. In the same way, we define $\exists$-homogeneity by considering the first-order formulas involving only the $\exists$ quantifier (see \ref{logic}).

In \cite{Nie03}, Nies proved that the free group $\mathbf{F}_2$ on two generators is $\exists$-homogeneous. Then, Perin and Sklinos, and independently Ould Houcine, proved that the free groups are homogeneous (see \cite{PS12} and \cite{OH11}). Moreover, Perin and Sklinos showed in \cite{PS12} that the fundamental group of a closed orientable surface of genus $\geq 2$ is not homogeneous, using a deep result of Sela. Later, Byron and Perin gave a complete characterization of freely-indecomposable torsion-free homogeneous hyperbolic groups, in terms of their JSJ decomposition. 

In this paper, we are mainly concerned with homogeneity in the class of finitely generated virtually free groups, i.e.\ finitely generated groups with a free subgroup of finite index. Note that there is no \textit{a priori} relation between homogeneity in a group and homogeneity in a finite index subgroup or in a finite extension. In the sequel, all virtually free groups are assumed to be finitely generated. 

We prove that all virtually free groups satisfy a slightly weaker form of homogeneity: a group $G$ is \textit{almost-homogeneous} if for any $k\geq 1$ and $\mathbf{u}\in G^k$, there exists an integer $N\geq 1$ such that the set of $k$-tuples of elements of $G$ having the same type as $\mathbf{u}$ is the union of $N$ orbits under the action of $\mathrm{Aut}(G)$, and $G$ is \textit{uniformly almost-homogeneous} if $N$ can be chosen independently from $\mathbf{u}$ and $k$. Note that $G$ is homogenous if and only if one can take $N=1$ in the previous definition.

\begin{te}\label{principal}Virtually free groups are uniformly almost-homogeneous.\end{te}

\begin{rque}In fact, virtually free groups are uniformly $\forall\exists$-almost-homogeneous (see \ref{logic}).\end{rque}

It is worth noting that the work of Perin and Sklinos (see \cite{PS12}) shows in fact that the fundamental group of an orientable hyperbolic closed surface is not almost-homogeneous.

In some cases, Theorem \ref{principal} above can be strengthened. For instance, in a work in progress, we prove that co-Hopfian virtually free groups are homogeneous. As an example, the group $\mathrm{GL}_2(\mathbb{Z})$ is homogeneous. We don't know whether or not virtually free groups are homogeneous in general. However, we shall prove that there exists a virtually free group that is not $\exists$-homogeneous. We build this example by exploiting a phenomenom that is specific to torsion (see below for further details). 

An important ingredient of the proof of Theorem \ref{principal} above is the so-called shortening argument, which generalizes to one-ended hyperbolic groups relative to a given subgroup, in the presence of torsion, using results of Guirardel \cite{Gui08} and Reinfeldt-Weidmann \cite{RW14} (generalizing previous work of Rips and Sela, see \cite{RS94}). In particular, the relative co-Hopf property holds for hyperbolic groups with torsion.

\begin{te}Let $G$ be a hyperbolic group and let $H$ be a finitely generated subgroup of $G$. Assume that $G$ is one-ended relative to $H$. Then every monomorphism of $G$ whose restriction to $H$ is the identity is an automorphism of $G$.\end{te} 

\begin{rque}In the case where $G$ is torsion-free and $H$ is trivial, this result was proved by Sela in \cite{Sel97}. For $G$ torsion-free and $H$ infinite, it was proved by Perin in his PhD thesis \cite{Per08} (see also \cite{Per11} and \cite{PS12}). In the case where $G$ has torsion and $H$ is finite, this theorem is due to Moioli (see his PhD thesis \cite{Moi13}).\end{rque}

\subsection*{$\exists$-homogeneity}\label{details}In \cite{Nie03}, Nies proved that the free group $\mathbf{F}_2$ is $\exists$-homogeneous. We stress that it is not known whether the free group $\mathbf{F}_n$ is $\exists$-homogeneous for $n\geq 3$. In Section \ref{contre-exemple}, we will give an example of a virtually free group that is not $\exists$-homogeneous (and we do not know whether this group is homogeneous). More precisely, we will prove the following result.                                                                                                                                                                         

\begin{prop}\label{contre-exemple2}
There exist a virtually free group $G=A\ast_C B$, with $A,B$ finite, and two elements $x,y\in G$ such that:
\begin{itemize}
\item[$\bullet$]there exists a monomorphism $G\hookrightarrow G$ that interchanges $x$ and $y$ (in particular, $x$ and $y$ have the same existential type);
\item[$\bullet$]no automorphism of $G$ maps $x$ to $y$.
\end{itemize}
\end{prop}

This is a new phenomenon, which does not appear in free groups, as shown by the proposition below (see \cite{OH11} Lemma 3.7, or Proposition \ref{libre} below).

\begin{prop}\label{introp}Let $x$ and $y$ be two elements of the free group $\mathbf{F}_n$. The two following statements are equivalent.
\begin{enumerate}
\item There exists a monomorphism $G\hookrightarrow G$ that sends $x$ to $y$, and a monomorphism $G\hookrightarrow G$ that sends $y$ to $x$.
\item There exists an automorphism of $G$ that sends $x$ to $y$.
\end{enumerate}
\end{prop}

The reason why the previous result holds in free groups is that every isomorphism between the free factors of $\mathbf{F}_n$ containing $x$ and $y$ respectively is the restriction of an ambient automorphism of $\mathbf{F}_n$. This fact fails in virtually free groups in general.

\subsection*{Generic homogeneity}

It is natural to wonder to what extent a given non-homogeneous hyperbolic group is far from being homogeneous. We shall prove that, in the sense of random walks, the $\mathrm{Aut}(G)$-orbit of a tuple in a hyperbolic group $G$ is determined by first-order logic. More precisely, a random tuple in a hyperbolic group has the property that if it has the same type as another tuple, then these two tuples are in the same orbit under the action of the automorphism group. We introduce the following definition.

\begin{de}Let $G$ be a group. A $k$-tuple $u\in G^k$ is said to be \textit{type-determined} if it has the following property: for every $k$-tuple $v\in G^k$, if $u$ and $v$ have the same type, then they are in the same $\mathrm{Aut}(G)$-orbit.\end{de} 

Let $G$ be a finitely generated group and let $\mu$ be a probability measure on $G$ whose support is finite and generates $G$. An element of $G$ arising from a random walk on $G$ of length $n$ generated by $\mu$ is called a \textit{random element} of length $n$. We define a \textit{random $k$-tuple} of length $n$ as a $k$-tuple of random elements of length $n$ arising from $k$ independant random walks. In Section \ref{generic}, we prove the following result.

\begin{te}\label{randomgeneric}Fix an integer $k\geq 1$. In a hyperbolic group, the probability that a random $k$-tuple of length $n$ is type-determined tends to one as $n$ tends to infinity.\end{te}

Theorem \ref{randomgeneric} is a consequence of the two following results. Say a tuple $u$ is \textit{rigid} if the group $G$ does not split non-trivially relative to $u$. 

\begin{prop}Fix an integer $k\geq 1$. In a hyperbolic group, the probability that a random $k$-tuple of length $n$ is rigid tends to one as $n$ tends to infinity.\end{prop}

\begin{prop}Fix an integer $k\geq 1$. Let $G$ be a hyperbolic group, and let $u\in G^k$. If $u$ is rigid, then $u$ is type-determined.\end{prop}

\begin{rque}In fact, we shall prove the following stronger result: if $u\in G^k$ is rigid, then it is $\exists$-\textit{type-determined}, meaning that any $k$-tuple $v$ with the same $\exists$-type as $u$ belongs to the same $\mathrm{Aut}(G)$-orbit.\end{rque}

The first proposition above relies on a result of Maher and Sisto proved in \cite{MS17} together with a work in progress of Guirardel and Levitt (see Section \ref{generic} for further details). The second proposition will be proved in Section \ref{generic}.

\subsection*{Strategy for proving almost-homogeneity}In Section \ref{motivation}, we prove that the group $\mathrm{SL}_2(\mathbb{Z})$ is homogeneous (more precisely, it is $\exists$-homogeneous). This example will serve as a model for proving almost-homogeneity of virtually free groups. Recall that $\mathrm{SL}_2(\mathbb{Z})$ splits as $\mathrm{SL}_2(\mathbb{Z})=A\ast_C B$ where $A\simeq \mathbb{Z}/4\mathbb{Z}$, $B\simeq\mathbb{Z}/6\mathbb{Z}$ and $C\simeq\mathbb{Z}/2\mathbb{Z}$. In particular, it is a virtually free group. Let us give a brief outline of the proof of the homogeneity of $G=\mathrm{SL}_2(\mathbb{Z})$. Let $u,v\in G^k$ be two $k$-tuples. Suppose that $u$ and $v$ have the same type.

\emph{Step 1.} By means of first-order logic, we prove that there exists a monomorphism $\phi : G \hookrightarrow G$ sending $u$ to $v$, and a monomorphism $\psi : G \hookrightarrow G$ sending $v$ to $u$. 

\emph{Setp 2.} We modify (if necessary) the monomorphism $\phi$ to get an automorphism $\sigma$ of $G$ that maps $u$ to $v$. 

We shall use a similar approach to prove that virtually free groups are uniformly almost-homogeneous. In Section \ref{mono}, we generalize the method used by Perin and Sklinos to prove the homogeneity of free groups, and we prove the following result (which can be compared with Step 1 above). 

\begin{prop}\label{propintro}
Let $G$ be a virtually free group, let $k\geq1$ be an integer and let $u,v\in G^k$. Let $U$ be the maximal one-ended subgroup of $G$ relative to $\langle u\rangle$, and let $V$ be the maximal one-ended subgroup of $G$ relative to $\langle v\rangle$. If $u$ and $v$ have the same type, then
\begin{itemize}
\item[$\bullet$]there exists an endomorphism $\phi$ of $G$ that maps $u$ to $v$ and whose restriction to $U$ is injective,
\item[$\bullet$]there exists an endomorphism $\psi$ of $G$ that maps $v$ to $u$ and whose restriction to $V$ is injective.
\end{itemize}
In fact, it is enough to suppose that $u$ and $v$ have the same $\forall\exists$-type.
\end{prop}

\begin{rque}Let us emphasize that Proposition \ref{propintro} together with Proposition \ref{introp} imply that finitely generated free groups are homogeneous. Indeed, for free groups, the existence of the endomorphisms $\phi$ and $\psi$ above is equivalent to the existence of an automorphism of $G$ sending $u$ to $v$, by Proposition \ref{introp}.\end{rque}

In the proof of the homogeneity of $\mathbf{F}_n$ and $\mathrm{SL}_2(\mathbb{Z})$, the existence of these endomorphisms $\phi$ and $\psi$ implies the existence of an automorphism of $G$ that sends $u$ to $v$. Unfortunately, the counterexample \ref{contre-exemple2} above shows that this step fails in the general case. In Section \ref{presque-homogénéité}, we circumvent this problem and prove that virtually free groups are uniformly almost-homogeneous. The key ingredient is the cocompactness of the Stallings deformation space of a virtually free group (see Proposition \ref{compacité} for a precise statement).

\subsection*{Acknowledgements}It is a pleasure to express my gratitude to my advisor Vincent Guirardel for his precious help and his many comments on previous versions of this paper.

\hypersetup{colorlinks=true, linkcolor=black}
\tableofcontents
\hypersetup{colorlinks=true, linkcolor=red}

\section{Preliminaries}

\subsection{First-order logic}\label{logic}

For detailed background on first-order logic, we refer the reader to \cite{Mar02}.

A first-order formula in the language of groups is a finite formula using the following symbols: $\forall$, $\exists$, $=$, $\wedge$, $\vee$, $\Rightarrow$, $\neq$, $1$ (standing for the identity element), ${}^{-1}$ (standing for the inverse), $\cdot$ (standing for the group multiplication) and variables $x,y,g,z\ldots$ which are to be interpreted as elements of a group. A variable is free if it is not bound by any quantifier $\forall$ or $\exists$. A sentence is a formula without free variables. A $\forall\exists$-formula is a formula of the form $\theta(\mathbf{x}):\forall\mathbf{y}\exists\mathbf{z}\varphi(\mathbf{x},\mathbf{y},\mathbf{z})$. An existential formula (or $\exists$-formula) is a formula in which the symbol $\forall$ does not appear. 

Given a formula $\theta(\mathbf{x})$ with $k\geq 0$ free variables, and a $k$-tuple $\mathbf{u}$ of elements of a group $G$, we say that $\mathbf{u}$ satisfies $\theta(\mathbf{x})$ if the statement $\theta(\mathbf{u})$ is true in $G$. 

Roughly speaking, a group $G$ is said to be homogeneous if the power of expression of first-order logic is strong enough to distinguish between two tuples that belong to different orbits under the action of the automorphism group of $G$. Here below are precise definitions.

\begin{de}Let $G$ be a group. We say that two $k$-tuples $\mathbf{u}$ and $\mathbf{v}$ of elements of $G$ have the same type if, for every first-order formula $\theta(\mathbf{x})$ with $k$ free variables, the statement $\theta(\mathbf{u})$ is true in $G$ if and only if the statement $\theta(\mathbf{v})$ is true in $G$. We use the notation $\mathrm{tp}(\mathbf{u})=\mathrm{tp}(\mathbf{v})$. In the same way, we say that two tuples have the same $\exists$-type (resp. $\forall\exists$-type) if they satisfy the same $\exists$-formulas (resp. $\forall\exists$-formulas). We use the notation $\mathrm{tp}_{\exists}(\mathbf{u})=\mathrm{tp}_{\exists}(\mathbf{v})$ (resp. $\mathrm{tp}_{\forall\exists}(\mathbf{u})=\mathrm{tp}_{\forall\exists}(\mathbf{v})$).\end{de}

\begin{de}Let $G$ be a group. We say that $G$ is homogeneous if, for every integer $k$ and for all $k$-tuples $\mathbf{u}$ and $\mathbf{v}$ having the same type, there exists an automorphism $\sigma$ of $G$ sending $\mathbf{u}$ to $\mathbf{v}$. In the same way, we define $\exists$-homogeneity by considering only $\exists$-formulas, and $\forall\exists$-homogeneity by considering only $\forall\exists$-formulas.\end{de}

\begin{de}Let $G$ be a group. We say that $G$ is almost-homogeneous if for every integer $k\geq 1$ and for every $k$-tuple $\mathbf{u}\in G^k$, \[\vert\lbrace v\in G^k \ \vert \ \mathrm{tp}(\mathbf{u})=\mathrm{tp}(\mathbf{v})\rbrace/\mathrm{Aut}(G)\vert < \infty. \]
We say that $G$ is uniformly almost-homogeneous if there exists an integer $N\geq 1$ such that for every integer $k\geq 1$ and every $k$-tuple $\mathbf{u}\in G^k$, \[\vert\lbrace v\in G^k \ \vert \ \mathrm{tp}(\mathbf{u})=\mathrm{tp}(\mathbf{v})\rbrace/\mathrm{Aut}(G)\vert \leq N.\]
In the same way, we define (uniform) $\exists$-almost-homogeneity by considering only $\exists$-formulas, and (uniform) $\forall\exists$-almost-homogeneity by considering only $\forall\exists$-formulas.\end{de}
\begin{rque}Note that a group is homogeneous if and only if it is uniformly almost-homogeneous with $N=1$.\end{rque}

\subsection{Virtually free groups and Stallings splittings}

A finitely generated group $G$ is virtually free if and only if it splits as a minimal graph of groups with finite vertex groups, or equivalently if it acts minimally on a simplicial tree by simplicial automorphisms, without inversions and with finite vertex stabilizers. Such a tree, endowed with the action of $G$, is called a \emph{Stallings tree} (or splitting) of $G$. The \emph{Stallings deformation space}, denoted by $\mathcal{D}(G)$, is the set of Stallings trees of $G$ up to equivariant isometry. 

Let $T$ be a Stallings tree of $G$, and let $e=[v,w]$ be an edge of $T$. Suppose that $G_v=G_e$ and that $v$ and $w$ are in distinct orbits. Collapsing every edge in the orbit of $e$ to a point produces a new Stallings tree $T'$ of $G$. We say that $T'$ is obtained from $T$ by an \emph{elementary collapse}, and that $T$ is obtained from $T'$ by an \emph{elementary expansion}. This elementary expansion of $T'$ introduces a new vertex group $G_v$ by using the isomorphism $G_w\simeq {G_w}\ast_{G_v} G_v$. These two operations are called \emph{elementary deformations}. A \emph{slide move} is defined as an elementary expansion followed by an elementary collapse. A Stallings tree $T$ of $G$ is said to be \emph{reduced} if there is no edge of the form $e=[v,w]$ with $G_v=G_e$, i.e.\ if one cannot perform any elementary collapse in $T$. Two Stallings trees of $G$ are connected by a sequence of elementary deformations. Two reduced Stallings trees of $G$ are connected by a sequence of slide moves. A vertex of $T$ is called \emph{redundant} if it has degree $2$. The tree $T$ is called \emph{non-redundant} if every vertex is non-redundant.

The following result is a particular case of Lemmas 2.20 and 2.22 in \cite{DG11}. For the definition of an isomorphism of graphs of groups, we refer the reader to Definition 2.19 in \cite{DG11}.

\begin{prop}\label{DG}Let $T$ and $T'$ be two Stallings trees of $G$. The two following assertions are equivalent.
\begin{enumerate}
\item The quotient graphs of groups $T/G$ and $T'/G$ are isomorphic.
\item There exist an automorphism $\sigma$ of $G$ and a $\sigma$-equivariant isometry $f : T \rightarrow T'$. 
\end{enumerate}
\end{prop}

In the latter case, we use the notation $T'=T^{\sigma}$. This is not ambiguous since we consider $\mathcal{D}(G)$ up to equivariant isometry. 

We shall need the following proposition that claims, in a sense, that $\mathcal{D}(G)$ is cocompact under the action of $\mathrm{Aut}(G)$.

\begin{prop}\label{compacité}Let $G$ be a virtually free group. There exist finitely many trees $T_1,\ldots, T_n$ in $\mathcal{D}(G)$ such that, for every non-redundant tree $T\in \mathcal{D}(G)$, there exist an automorphism $\sigma$ of $G$ and an integer $1\leq k\leq n$ such that $T=T_k^{\sigma}$.
\end{prop}

Before proving this proposition, we need two lemmas.

\begin{lemme}\label{premierlemme}Let $G$ be a virtually free group. There exist finitely many reduced trees $T_1,\ldots, T_n$ in $\mathcal{D}(G)$ such that, for every reduced tree $T\in \mathcal{D}(G)$, there exists an integer $1\leq k\leq n$ such that $T=T_k^{\sigma}$ for some $\sigma\in\mathrm{Aut}(G)$.\end{lemme}

\begin{proof}
Given two reduced Stallings trees $T$ and $T'$ of $G$, one can pass from $T$ to $T'$ by a sequence of slide moves. Consequently, all reduced Stallings trees of $G$ have the same number of orbits of vertices, say $p$, and the same number of orbits of edges, say $q$. Moreover, all reduced Stallings trees have the same vertex groups. Let $r$ be the maximal order of such a vertex group. Now, observe that there are only finitely many isomorphism classes of graphs of finite groups with $p$ vertices, $q$ edges, and whose vertex groups have order $\leq r$. By Proposition \ref{DG}, it is sufficient to conclude.\end{proof}

\begin{lemme}\label{secondlemme}Let $G$ be a virtually free group, and let $T$ be a reduced Stallings tree of $G$. There are only finitely many non-redundant Stallings trees that can be obtained from $T$ by a sequence of elementary expansions.\end{lemme}

\begin{proof}We claim that for any sequence of Stallings trees $(T_k)_{k\in\mathbb{N}}$ such that $T_0=T$ and such that $T_{k+1}$ is obtained from $T_k$ by an elementary expansion, the tree $T_k$ is redundant for $k$ sufficiently large. This is enough to conclude. Indeed, given any Stallings tree $T'$ of $G$, there are only finitely many Stallings trees that can be obtained from $T'$ by an elementary expansion, so the lemma follows from the previous claim combined with König's lemma.

We now prove the claim. Let us show that for $k$ large enough, one can find an arbitrarily long path in $T_k/G$ in which each vertex has degree 2. Let $V_k$ denote the set of vertices of $T_k/G$. For each vertex $v\in V_k$, we denote by $\deg(v)$ the degree of $v$ in $T_k/G$. We make the following observation: the sum $S_k=\sum_{v\in V_k}(\deg(v)-2)$ is constant along the sequence $(T_k)$, as illustrated below.

\begin{figure}[!h]
\includegraphics[scale=0.7]{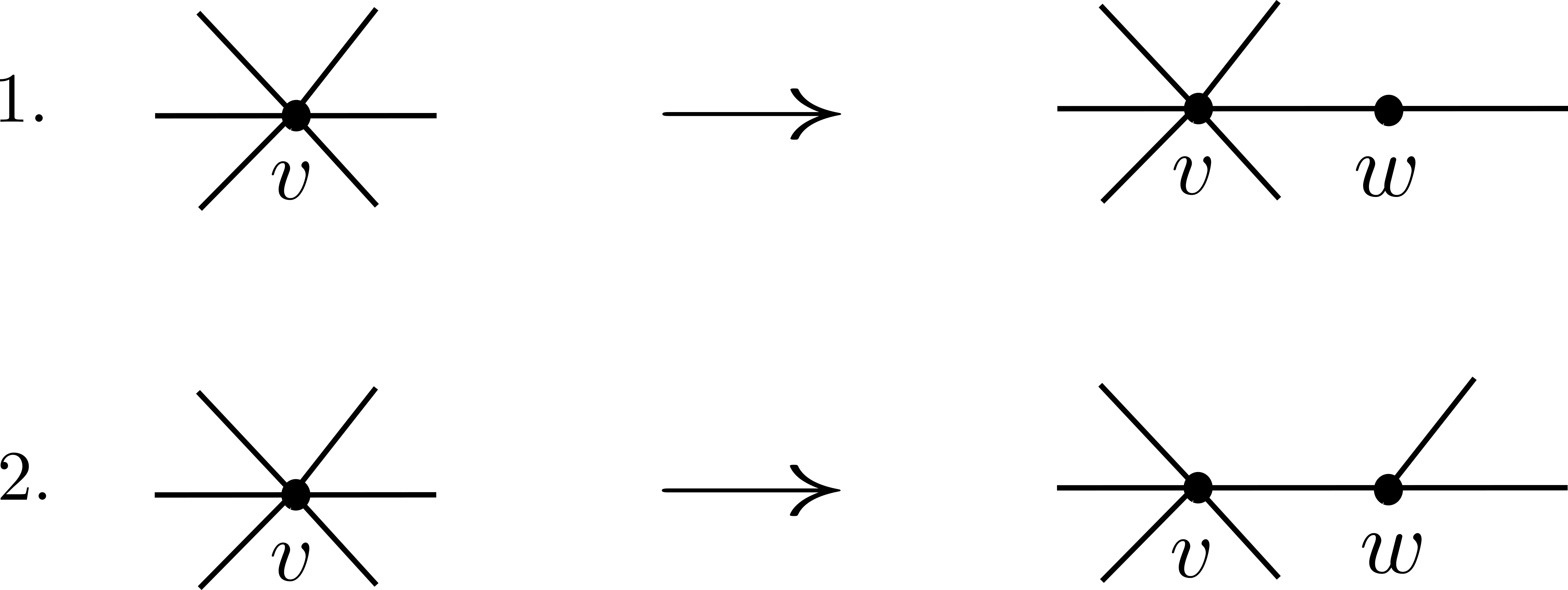}
\caption{The vertex $v$ has degree 6 in $T_k/G$. In the first case, the elementary expansion gives rise to a vertex $w$ of degree $2$ in $T_{k+1}/G$. In the second case, the elementary expansion gives rise to a vertex $w$ of degree $3$ in $T_{k+1}/G$, and the degree of $v$ decreases to 5. In both cases, the sums (6-2)+(2-2) and (5-2)+(3-2) are equal to 4.}
\label{expansion}
\end{figure}

Moreover, the number of vertices of degree $1$ in $T_k/G$ is bounded by the number of conjugacy classes of maximal finite subgroups of $G$. It follows that the number of vertices of degree $\geq 3$ in $T_k/G$ is bounded independently from $k$. Therefore, for any $m\geq 1$, there exists an integer $k$ and a path $[v_1,v_2,\ldots,v_m]$ in $T_k/G$ such that, for every $\ell\in\llbracket 1,m\rrbracket$, $v_{\ell}$ has degree 2 in $T_k/G$ and $G_{v_{\ell+1}}\leq G_{v_\ell}$. Now, taking for $m$ the maximal order of a finite subgroup of $G$, one has necessarily $G_{v_{\ell}}=G_{v_{\ell+1}}$ for some $\ell\in\llbracket 1,m\rrbracket$. Hence, each preimage of $v_{\ell}$ in $T_k$ has degree 2, so $T_k$ is redundant.\end{proof}

We can now prove Proposition \ref{compacité}.

\begin{proof}Let $T_1,\ldots, T_n$ be the reduced Stallings trees of $G$ given by Lemma \ref{premierlemme}. For each $T_k$ in this finite set, Lemma \ref{secondlemme} provides us with a finite collection of non-redundant trees $T_{k,1},\ldots ,T_{k,r_k}$. Let $\mathcal{T}$ be the union of these sets $\lbrace T_{k,1},\ldots ,T_{k,r_k}\rbrace$, for $k$ from $1$ to $n$. Let $T$ be a non-redundant tree in $\mathcal{D}(G)$. We claim that $T$ belongs to the orbit of an element of $\mathcal{T}$ under the action of $\mathrm{Aut}(G)$. First, note that $T$ collapses onto a reduced Stallings tree $T'$. In other words, $T$ is obtained from $T'$ by a sequence of elementary expansions. By Lemma \ref{premierlemme}, $T'=T_k^{\sigma}$ for some automorphism $\sigma$ of $G$ and some $k\in\llbracket 1,n\rrbracket$. Thanks to Lemma \ref{secondlemme}, $T=T_{k,i}^{\sigma}$ for some $i\in\llbracket 1,r_k\rrbracket$.\end{proof}

\subsection{The JSJ decomposition and the modular group}\label{25}

\begin{de}\label{FBO}A group $G$ is called a \emph{finite-by-orbifold group} if it is an extension \[1\rightarrow F\rightarrow G \rightarrow \pi_1(\mathcal{O})\rightarrow 1\]where $\mathcal{O}$ is a compact hyperbolic 2-orbifold possibly with boundary, and $F$ is an arbitrary finite group called the fiber. We call \emph{extended boundary subgroup} of $G$ the preimage in $G$ of a boundary subgroup of $\pi_1(\mathcal{O})$. We define in the same way \emph{extended conical subgroups}. In the case where $\mathcal{O}$ has only conical singularities, i.e.\ has no mirrors, we say that $G$ is a conical finite-by-orbifold group.
\end{de}

\begin{de}\label{QH}A vertex $v$ of a graph of groups is said to be \emph{quadratically hanging} (denoted by QH) if its stabilizer $G_v$ is a finite-by-orbifold group $F\rightarrow G \rightarrow \pi_1(\mathcal{O})$ such that $\mathcal{O}$ has non-empty boundary, and such that any incident edge group is finite or contained in an extended boundary subgroup. We also say that $G_v$ is QH. 
\end{de}

We denote by $\mathcal{Z}$ the class of groups that are either finite or virtually cyclic with infinite center.

Let $G$ be a hyperbolic group and let $H$ be a subgroup of $G$. Assume that $G$ is one-ended relative to $H$. The group $G$ has a canonical JSJ decomposition $\Delta$ over $\mathcal{Z}$ relative to $H$. We refer the reader to \cite{GL16} for a construction of $\Delta$ using the tree of cylinders. Below are the properties of $\Delta$ that will be useful in the sequel.
\begin{itemize}
\item[$\bullet$]The graph $\Delta$ is bipartite, with every edge joining a vertex carrying a virtually cyclic group to a vertex carrying a non-virtually-cyclic group.
\item[$\bullet$]There are two kinds of vertices of $\Delta$ carrying a non-cyclic group: rigid ones, and QH ones. If $v$ is a QH vertex of $\Delta$, every incident edge group $G_e$ coincides with an extended boundary subgroup of $G_v$. Moreover, given any extended boundary subgroup $B$ of $G_v$, there exists a unique incident edge $e$ such that $G_e=B$.
\item[$\bullet$]The action of $G$ on the associated Bass-Serre tree $T$ is acylindrical in the following strong sense: if an element $g\in G$ of infinite order fixes a segment of length $\geq 2$ in $T$, then this segment has length exactly 2 and its midpoint has virtually cyclic stabilizer.
\item Let $v$ be a vertex of $T$, and let $e,e'$ be two distinct edges incident to $v$. If $G_v$ is not virtually cyclic, then the group $\langle G_e,G_{e'}\rangle$ is not virtually cyclic.
\end{itemize}

\begin{de}\label{modul}Let $G$ be a hyperbolic group one-ended relative to a subgroup $H<G$. We denote by $\mathrm{Aut}_H(G)$ the subgroup of $\mathrm{Aut}(G)$ consisting of all automorphisms whose restriction to $H$ is the conjugacy by an element of $G$.

The \emph{modular group} $\mathrm{Mod}_H(G)$ of $G$ relative to $H$ is the subgroup of $\mathrm{Aut}_H(G)$ consisting of all automorphisms $\sigma$ satisfying the following conditions:
\begin{itemize}
\item[$\bullet$]the restriction of $\sigma$ to each non-QH vertex group of the $\mathcal{Z}$-JSJ splitting of $G$ relative to $H$ is the conjugacy by an element of $G$,
\item[$\bullet$]the restriction of $\sigma$ to each finite subgroup of $G$ is the conjugacy by an element of $G$,
\item[$\bullet$]$\sigma$ acts trivially on the underlying graph of the $\mathcal{Z}$-JSJ splitting relative to $H$.
\end{itemize}
\end{de}

\begin{rque}Rather than defining $\mathrm{Mod}_H(G)$ as above by imposing conditions on the action on vertex groups, one could define it by giving generators: twists around edges, and certain automorphisms of vertex groups (see for example \cite{Per11} and \cite{RW14}). These two definitions yield the same subgroup of $\mathrm{Aut}_H(G)$, up to finite index. We refer the reader to \cite{GL15} Section 5 for further discussion about this issue.\end{rque}

The following result is an immediate consequence of Theorem 4.6 in \cite{GL15}.

\begin{te}\label{indice fini}The modular group $\mathrm{Mod}_H(G)$ has finite index in $\mathrm{Aut}_H(G)$.\end{te}

\begin{proof}According to Theorem 4.6 in \cite{GL15}, the modular group $\mathrm{Mod}_H(G)/\mathrm{Inn}(G)$ contains a subgroup, denoted by $\mathrm{Out}^1(G,\lbrace H\rbrace^{(t)})$ in \cite{GL15}, whose index is finite in the group $\mathrm{Aut}_H(G)/\mathrm{Inn}(G)$, denoted by $\mathrm{Out}(G,\lbrace H\rbrace^{(t)})$ in \cite{GL15}.\end{proof} 

\subsection{Related homomorphisms and preretractions}\label{32}

In the sequel, we denote by $\mathrm{ad}(g)$ the inner automorphism $h\mapsto ghg^{-1}$.

\begin{de}[Related homomorphisms]\label{reliés2}
Let $G$ be a hyperbolic group and let $H$ be a subgroup of $G$. Assume that $G$ is one-ended relative to $H$. Let $G'$ be a group. Let $\Delta$ be the canonical JSJ splitting of $G$ over $\mathcal{Z}$ relative to $H$. Let $\phi$ and $\phi'$ be two homomorphisms from $G$ to $G'$. We say that $\phi$ and $\phi'$ are $\Delta$-related if the two following conditions hold:
\begin{itemize}
\item[$\bullet$]for every non-QH vertex $v$ of $\Delta$, there exists an element $g_v\in G'$ such that \[{\phi'}_{\vert G_v}=\mathrm{ad}({g_v})\circ \phi_{\vert G_v};\]
\item[$\bullet$]for every finite subgroup $F$ of $G$, there exists an element $g\in G'$ such that \[{\phi'}_{\vert F}=\mathrm{ad}({g})\circ \phi_{\vert F}.\]
\end{itemize}
\end{de}

\begin{de}[Preretraction]\label{pre}
Let $G$ be a hyperbolic group, and let $H$ be a subgroup of $G$. Assume that $G$ is one-ended relative to $H$. Let $\Delta$ be the canonical JSJ splitting of $G$ over $\mathcal{Z}$ relative to $H$. A preretraction of $G$ is an endomorphism of $G$ that is $\Delta$-related to the identity map. More generally, if $G$ is a subgroup of $G'$, a preretraction from $G$ to $G'$ is a homomorphism that is $\Delta$-related to the inclusion of $G$ into $G'$.
\end{de}

The following easy lemma shows that being $\Delta$-related can be expressed in first-order logic.

\begin{lemme}\label{deltarelies}
Let $G$ be a hyperbolic group and let $H$ be a subgroup of $G$. Assume that $G$ is one-ended relative to $H$. Let $G'$ be a group. Let $\Delta$ be the canonical JSJ splitting of $G$ over $\mathcal{Z}$ relative to $H$. Let $\lbrace g_1,\ldots ,g_n\rbrace$ be a generating set of $G$. There exists an existential formula $\theta(x_1,\ldots , x_{2n})$ with $2n$ free variables such that, for every $\phi,\phi'\in \mathrm{Hom}(G,G')$, $\phi$ and $\phi'$ are $\Delta$-related if and only if $G'$ satisfies $\theta\left(\phi(g_1),\ldots , \phi(g_n),\phi'(g_1),\ldots ,\phi'(g_n)\right).$
\end{lemme}

\begin{proof}
Firstly, remark that there exist finitely many (say $p\geq 1$) conjugacy classes of finite subgroups of QH vertex groups of $\Delta$ (indeed, a QH vertex group possesses finitely many conjugacy classes of finite subgroups, and $\Delta$ has finitely many vertices). Denote by $F_1,\ldots ,F_p$ a system of representatives of those conjugacy classes. Denote by $R_1,\ldots ,R_m$ the non-QH vertex groups of $\Delta$. Remark that these groups are finitely generated since $G$ and the edge groups of $\Delta$ are finitely generated. Denote by $\lbrace A_i\rbrace_{1\leq i\leq p+m}$ the union of $\lbrace F_i\rbrace_{1\leq i\leq p}$ and $\lbrace R_i\rbrace_{1\leq i\leq m}$. For every $i\in\llbracket 1,m+p\rrbracket$, let $\lbrace a_{i,1},\ldots ,a_{i,k_i}\rbrace$ be a finite generating set of $A_i$. For every $i\in\llbracket 1,m+p\rrbracket$ and $j\in\llbracket 1,k_i\rrbracket$, there exists a word $w_{i,j}$ in $n$ letters such that $a_{i,j}=w_{i,j}(g_1,\ldots ,g_n)$. Let \[\theta(x_1,\dots,x_{2n}):\exists u_1\ldots\exists u_m \bigwedge_{i=1}^{m+p}\bigwedge_{j=1}^{k_i}w_{i,j}(x_1,\ldots ,x_n)=u_iw_{i,j}(x_{n+1},\ldots ,x_{2n}){u_i}^{-1}. \]

Since $\phi(a_{i,j})=w_{i,j}(\phi(g_1),\ldots ,\phi(g_n))$ and $\phi'(a_{i,j})=w_{i,j}\left(\phi'(g_1),\ldots ,\phi'(g_n)\right)$ for every $i\in\llbracket 1,m+p\rrbracket$ and $j\in\llbracket 1,k_i\rrbracket$, the homomorphisms $\phi$ and $\phi'$ are $\Delta$-related if and only if the sentence $\theta\left(\phi(g_1),\ldots , \phi(g_n),\phi'(g_1),\ldots ,\phi'(g_n)\right)$ is satisfied by $G'$.
\end{proof}

The proof of the following lemma is identical to that of Proposition 7.2 in \cite{And18}.

\begin{lemme}\label{lemmeperin}Let $G$ be a hyperbolic group. Suppose that $G$ is one-ended relative to a subgroup $H$. Let $\Delta$ be the $\mathcal{Z}$-JSJ splitting of $G$ relative to $H$. Let $\phi$ be a preretraction of $G$. If $\phi$ sends every QH group isomorphically to a conjugate of itself, then $\phi$ is injective.
\end{lemme}

\subsection{Centered graph of groups}\label{centered}

\begin{de}[Centered graph of groups]\label{graphecentre}A graph of groups over $\mathcal{Z}$, with at least two vertices, is said to be centered if the following conditions hold:
\begin{itemize}
\item[$\bullet$]the underlying graph is bipartite, with a QH vertex $v$ such that every vertex different from $v$ is adjacent to $v$;
\item[$\bullet$]every incident edge group $G_e$ coincides with an extended boundary subgroup or with an extended conical subgroup of $G_v$ (see Definition \ref{FBO});
\item[$\bullet$]given any extended boundary subgroup $B$, there exists a unique incident edge $e$ such that $G_e$ is conjugate to $B$ in $G_v$;
\item[$\bullet$]if an element of infinite order fixes a segment of length $\geq 2$ in the Bass-Serre tree of the splitting, then this segment has length exactly 2 and its endpoints are translates of $v$.
\end{itemize}
The vertex $v$ is called the central vertex.
\end{de}

\begin{figure}[!h]
\includegraphics[scale=0.4]{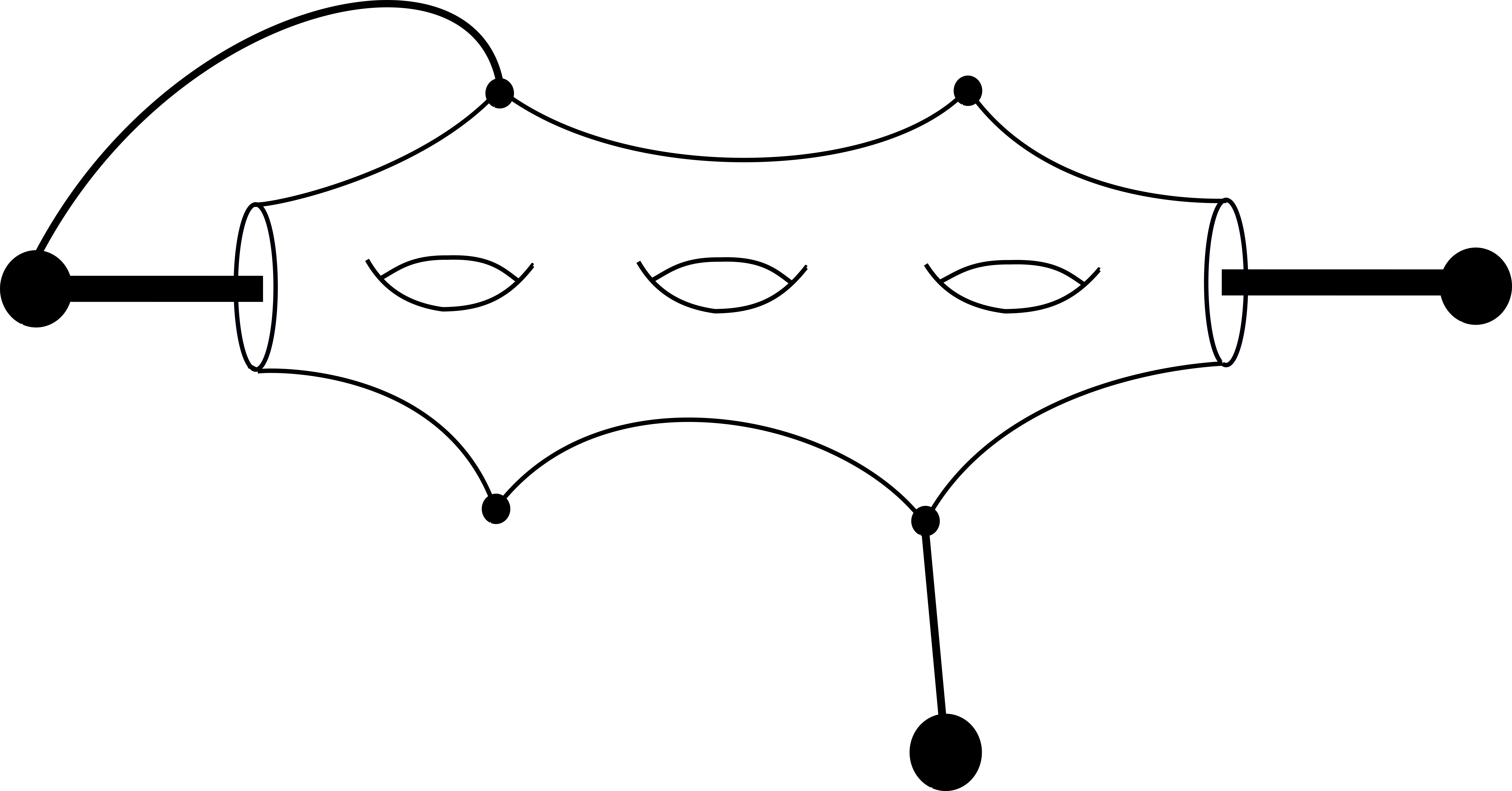}
\caption{A centered graph of groups. Edges with infinite stabilizer are depicted in bold.}
\end{figure}

\begin{de}[Related homomorphisms]\label{reliés}
\normalfont
Let $G$ and $G'$ be two groups. Suppose that $G$ possesses a centered splitting $\Delta$, with central vertex $v$. Let $\phi$ and $\phi'$ be two homomorphisms from $G$ to $G'$. We say that $\phi$ and $\phi'$ are $\Delta$-related if the two following conditions hold:
\begin{itemize}
\item[$\bullet$]for every vertex $w\neq v$, there exists an element $g_w\in G'$ such that \[{\phi'}_{\vert G_w}=\mathrm{ad}(g_w)\circ \phi_{\vert G_w};\]
\item[$\bullet$]for every finite subgroup $F$ of $G$, there exists an element $g\in G'$ such that \[{\phi'}_{\vert F}=\mathrm{ad}(g)\circ \phi_{\vert F}.\]
\end{itemize}
\end{de}

\begin{de}[Preretraction]\label{special}Let $G$ be a hyperbolic group, and let $\Delta$ be a centered splitting of $G$. Let $v$ be the central vertex of $\Delta$. An endomorphism $\phi$ of $G$ is called a preretraction if it is $\Delta$-related to the identity of $G$ in the sense of the previous definition. A preretraction is said to be non-degenerate if if does not send $G_v$ isomorphically to a conjugate of itself.\end{de}

A set of disjoint simple closed curves $S$ on a conical orbifold is said to be \textit{essential} if its elements are non null-homotopic, two-sided, non boundary-parallel, pairwise non parallel, and represent elements of infinite order (in other words, no curve of $S$ circles a singularity). Let $G$ be a group with a centered splitting $\Delta$, with central vertex $v$. The stabilizer $G_v$ of $v$ is a conical finite-by-orbifold group \[F\hookrightarrow G_v \overset{q}{\twoheadrightarrow}\pi_1(\mathcal{O}).\]Let $\phi$ be an endomorphism of $G$. Let $S$ be an essential set of curves on $\mathcal{O}$. An element $\alpha\in S$ is said to be \textit{pinched} by $\phi$ if $\phi\left(q^{-1}(\alpha)\right)$ is finite. The set $S$ is called a \textit{maximal pinched set} for $\phi$ if each element of $S$ is pinched by $\phi$, and if $S$ is maximal for this property. Note that $S$ may be empty.

The proof of the following lemma is identical to that of Lemma 7.4 in \cite{And18}.

\begin{lemme}\label{magique}Let $G$ be a hyperbolic group, and let $\Delta$ be a centered splitting of $G$. Let $v$ be the central vertex of $\Delta$ and let $\phi$ be a non-degenerate $\Delta$-preretraction of $G$. Let $S$ be a maximal pinched set for $\phi$ and let $\Delta(S)$ be the splitting of $G$ obtained from $\Delta$ by replacing the vertex $v$ by the splitting of $G_v$ dual to $S$. Let $v_1,\ldots ,v_n$ be the vertices of $\Delta(S)$ coming from the splitting of $G_v$. We denote by $\mathcal{S}$ the set of edges of $\Delta(S)$ corresponding to the maximal pinched set $S$, and we denote by $\mathcal{F}$ the set of edges of $\Delta(S)$ whose stabilizer is finite. Let $W$ be a connected component of $\Delta(S)\setminus (\mathcal{S}\cup \mathcal{F})$, and let $G_W$ be its stabilizer. Then $\phi(G_W)$ is elliptic in the Bass-Serre tree $T_{\Delta}$ of $\Delta$. Moreover, $W$ contains at most one vertex $w$ different from the vertices $v_k$ coming from the central vertex $v$. If it contains one, then $\phi(G_W)=\phi(G_w)$. See Figure \ref{image} below.\end{lemme}

\begin{figure}[!h]
\includegraphics[scale=0.7]{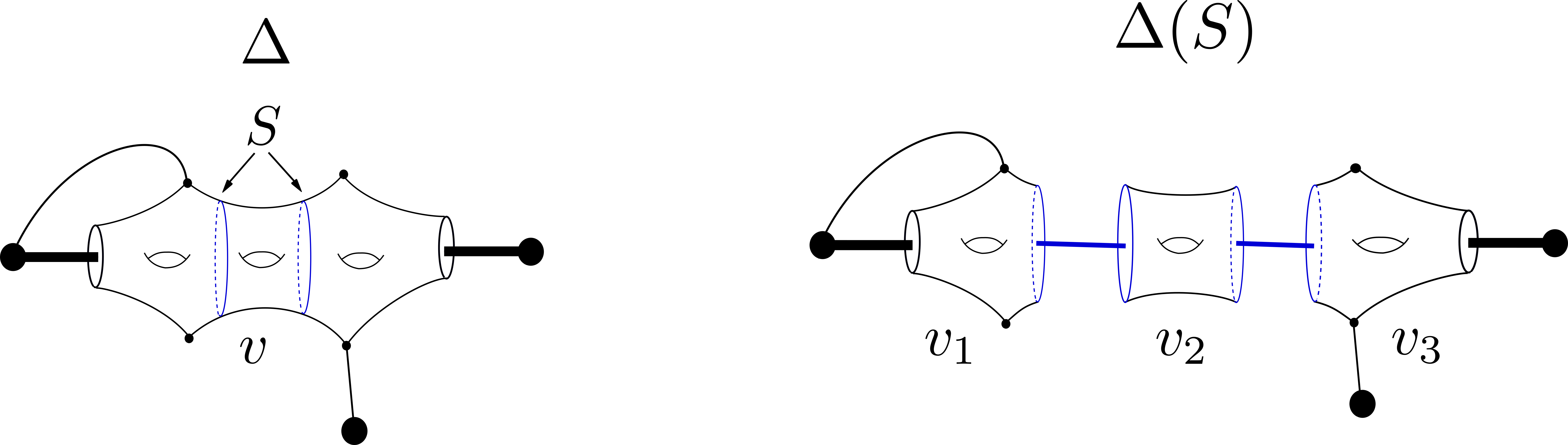}
\caption{On the left, the centered splitting $\Delta$ of $G$. On the right, the splitting $\Delta(S)$ of $G$ obtained from $\Delta$ by replacing the central vertex $v$ by the splitting of $G_v$ dual to $S$.}
\label{image}
\end{figure}

\newpage

\subsection{The shortening argument}

Let $\Gamma$ and $G$ be two hyperbolic groups. A sequence of homomorphisms $(\phi_n:G\rightarrow\Gamma )_{n\in\mathbb{N}}$ is said to be \emph{stable} if, for any $g\in G$, either $\phi_n(g)=1$ for almost all $n$ or $\phi_n(g)\neq 1$ for almost all $n$. The \emph{stable kernel} of the sequence is defined as $\lbrace g\in G \ \vert \ \phi_n(g)=1 \text{ for almost all } n \rbrace$. 

Let $S$ denote a finite generating set of $G$, and let $(X,d)$ denote the Cayley graph of $\Gamma$ (for a given finite generating set). For any $\phi\in \mathrm{Hom}(G,\Gamma)$, we define the \emph{length} of $\phi$ as $\ell(\phi)=\max_{s\in S}d(1,\phi(s))$. 

Let $H$ be a finitely generated subgroup of $G$. Assume that $G$ is one-ended relative to $H$ (meaning that $G$ does not split as $A\ast_C B$ or $A\ast_C$ with $C$ finite and $H$ contained in $A$). The \emph{shortening argument} (Proposition \ref{propo} below) asserts that, given a stable sequence of pairwise distinct homomorphisms $(\phi_n : G \rightarrow \Gamma)_{n\in\mathbb{N}}$ such that, for every $n$, $\phi_n$ coincides with $\phi_0$ on $H$ up to conjugacy by an element of $\Gamma$,
\begin{itemize}
\item[$\bullet$]either the sequence has non-trivial stable kernel,
\item[$\bullet$]or one can shorten $\phi_n$ for every $n$ large enough, meaning that there exists a modular automorphism $\sigma_n\in \mathrm{Aut}_H(G)$ and an element $\gamma_n\in \Gamma$ such that \[\ell(\mathrm{ad}(\gamma_n)\circ \phi_n\circ \sigma_n)<\ell(\phi_n).\]
\end{itemize}

This result has strong consequences on the structure of $\mathrm{Hom}(G,\Gamma)$ (see Theorems \ref{sa} and \ref{short} below). In particular, when $G=\Gamma$, the shortening argument implies that $G$ is co-Hopfian relative to $H$.

Let $\psi: H \rightarrow \Gamma$ be a homomorphism. In the sequel, we denote by $E(\psi)$ the subset of $\mathrm{Hom}(G,\Gamma)$ composed of all homomorphisms which coincide with $\psi$ on $H$ up to conjugation by an element of $\Gamma$. The group $\Gamma \times \mathrm{Aut}_H(G)$ acts on $E(\psi)$ by $(\gamma,\sigma)\cdot\phi=\mathrm{ad}(\gamma)\circ \phi\circ \sigma$.  We say that a homomorphism $\phi\in E(\psi)$ is \emph{short} (with respect to $\psi$ and $H$) if it has minimal length among its orbit under the action of $\Gamma\times \mathrm{Aut}_H(G)$. 

An action of $(G,H)$ on a tree $T$ is an action of $G$ on $T$ such that $H$ fixes a point. In the sequel, if $X\subset T$, $\mathrm{Stab}(X)$ denotes the pointwise stabilizer of $X$. An arc $I$ is said to be unstable is there exists an arc $J\subset I$ such that $\mathrm{Stab}(I)\subsetneq\mathrm{Stab}(J)$. We shall need the following theorem, proved by Guirardel in \cite{Gui08} (Theorem 5.1), which enables us to decompose actions on real trees into tractable building blocks, under certain conditions.

\begin{te}\label{guirardel}Let $G$ be a finitely generated group. Let $H$ be a subgroup of $G$. Consider a minimal and non-trivial action of $(G,H)$ on an $\mathbb{R}$-tree $T$ by isometries. Assume that
\begin{enumerate}
\item $T$ satisfies the ascending chain condition: for any decreasing sequence of arcs $I_1\supset I_2\supset\ldots$ whose lengths converge to $0$, the sequence of their pointwise stabilizers $\mathrm{Stab}(I_1)\subset\mathrm{Stab}(I_2)\subset\ldots$ stabilizes.
\item For any unstable arc $I\subset T$, 
\begin{enumerate}
\item $\mathrm{Stab}(I)$ is finitely generated,
\item $\forall g\in G$, $\mathrm{Stab}(I)^g\subset \mathrm{Stab}(I)\Rightarrow\mathrm{Stab}(I)^g = \mathrm{Stab}(I)$.
\end{enumerate}
\end{enumerate}
Then either $(G,H)$ splits over the stabilizer of an unstable arc, or over the stabilizer of an infinite tripod, or $T$ has a decomposition into a graph of actions where each vertex action is either
\begin{enumerate}
\item simplicial: $G_v\curvearrowright Y_v$ is a simplicial action on a simplicial tree;
\item of Seifert type: the vertex action $G_v\curvearrowright Y_v$ has kernel $N_v$, and the
faithful action $G_v/N_v\curvearrowright Y_v$ is dual to an arational measured foliation on a closed conical 2-orbifold with boundary;
\item axial: $Y_v$ is a line, and the image of $G_v$ in $\mathrm{Isom}(Y_v)$ is a finitely generated group acting with dense orbits on $Y_v$.
\end{enumerate}
\end{te}

We will need the two following results.

\begin{te}\label{shorten1}Let $G$ be a finitely generated group and let $H$ be a subgroup of $G$. Suppose that $(G,H)$ acts on a real tree $(T,d)$ that decomposes as a graph of actions $\mathcal{G}$. Let us denote by $\Delta_{\mathcal{G}}$ the corresponding splitting of $G$. Fix a point $x\in T$. Let $S$ be a finite generating set for $G$. There exists an element $\sigma\in \mathrm{Aut}_H(G)$ such that for every $s\in S$, the following holds:
\begin{itemize}
\item[$\bullet$]if the geodesic segment $[x,s\cdot x]$ intersects non-trivially a surface or an axial component, then $d(x,\sigma(s)\cdot x)<d(x,s\cdot x)$;
\item[$\bullet$]if not, then $\sigma(s)=s$.
\end{itemize}
\end{te}

For a proof of the result above, we refer the reader to \cite{Per08} for the torsion-free case (Theorems 5.12 and 5.17) and \cite{RW14} for the general case (Theorems 4.8 and 4.15).

\begin{rque}\label{reqimportante}In \cite{Per08} and \cite{RW14}, Theorem \ref{shorten1} is stated for $\mathrm{Mod}(\Delta_{\mathcal{G}})$ instead of $\mathrm{Aut}_H(G)$, but one easily sees that $\mathrm{Mod}(\Delta_{\mathcal{G}})$ is a subgroup of $\mathrm{Aut}_H(G)$, because $H$ fixes a point in $T$.\end{rque}

\begin{te}\label{shorten2}Let $\omega$ be a non-principal ultrafilter. Let $G$ be a finitely generated group and let $S$ be a finite generating set for $G$. Let $\Gamma$ be a hyperbolic group, and denote by $(X,d)$ the Cayley graph of $\Gamma$ (for a given generating set). Let $(\phi_n : G \rightarrow\Gamma)_n$ be a sequence of homomorphisms, and let $\lambda_n=\max_{s\in S}d(1,\phi_n(s))$. Assume that the sequence of pointed metric spaces $(X,d_n=d/\lambda_n,1)_n$ $\omega$-converges to a pointed real tree $(T,d_{\omega},x)$, and that this tree decomposes as a graph of actions $\mathcal{G}$. Let $\Delta_{\mathcal{G}}$ be the splitting of $G$ associated with $\mathcal{G}$, and let $\Lambda_{\mathcal{G}}$ be the splitting of $G$ obtained from $\Delta_{\mathcal{G}}$ by replacing each vertex $v$ of simplicial type by the corresponding splitting of $G_v$. The edges coming from this refinement are called simplicial. For every simplicial edge $e$, there exists a Dehn twist $\sigma$ around $e$, together with a sequence of integers $(m_n)_n$, such that for every $s\in S$
\begin{itemize}
\item[$\bullet$]if the geodesic segment $[x,s\cdot x]$ contains $e$, then $d_n(1,\phi_n\circ\sigma^{m_n}(s))<d_n(1,\phi_n(s))$ for every $n$ large enough, and
\item[$\bullet$]if the geodesic segment $[x,s\cdot x]$ does not contain $e$, then $d_n(1,\phi_n\circ\sigma^{m_n}(s))=d_n(1,\phi_n(s))$ for every $n$.
\end{itemize}
\end{te}

For a proof of the previous result, we refer the reader to \cite{Per08} Theorem 5.22 for the torsion-free case, and \cite{RW14} Corollary 4.20 for the general case.

\begin{rque}\label{reqimportante2}If a subgroup $H$ of $G$ fixes a point in $T$, then all Dehn twists around simplicial edges belong to $\mathrm{Aut}_H(G)$.\end{rque}

\begin{prop}\label{propo}Let $G$ and $\Gamma$ be hyperbolic groups. Let $H$ be a finitely generated subgroup of $G$. Suppose that $G$ is one-ended relative to $H$. Let $\psi : H \hookrightarrow \Gamma$ be a monomorphism. Let $(\phi_n:G\rightarrow\Gamma )\in E(\psi)^{\mathbb{N}}$ be a stable sequence of distinct short homomorphisms. Then the stable kernel of the sequence is non-trivial.\end{prop}

\begin{proof}
Assume towards a contradiction that the stable kernel is trivial. We shall find a contradiction as follows: first, we prove that the group $G$ acts fixed-point-freely on a real tree $T$. Secondly, Theorem \ref{guirardel} provides us with a decomposition of $T$ into a graph of actions. Last, we can prove that the homomorphisms $\phi_n$ can be shortened, using \cite{RW14}. This is a contradiction since the $\phi_n$ are supposed to be short.

Let $S$ be a finite generating set of $G$. Let $(X,d)$ be the Cayley graph of $\Gamma$ (for a given finite generating set of $\Gamma$). Since the homomorphisms $\phi_n$ are short, the point $1\in X$ minimizes the function $x\in X \mapsto\max_{s\in S}d(x,\phi_n(s)\cdot x)\in\mathbb{N}$, for each $n$. Let $\lambda_n=\max_{s\in S}d(1,\phi_n(s)\cdot 1)$. 

Let us observe that $\lambda_n$ goes to infinity as $n$ goes to infinity. Indeed, the homomorphisms $\phi_n$ are pairwise distinct, and for every integer $R$ the set \[\lbrace \phi\in\mathrm{Hom}(G,\Gamma) \ \vert \ \phi(S)\subset B_{\Gamma}(1,R)\rbrace\] is finite, since $\phi$ is determined by the finite set $\phi(S)$, and $\#B_{\Gamma}(1,R)<+\infty$.

Let $\omega$ be a non-principal ultrafilter. The sequence of metric spaces $((X,d_n:=d/\lambda_n))_{n\in\mathbb{N}}$ $\omega$-converges to a real tree $(T=X_{\omega},d_{\omega})$. Let $x=(1)_n\in T=X_{\omega}$. We claim that $G$ acts fixed-point-freely on $T$. Towards a contradiction, suppose that $G$ fixes a point $y=(y_n)_n\in T$. Since the generating set $S$ of $G$ is finite, there exists an element $s\in S$ such that $\omega(\lbrace n \ \vert \ d(1, \phi_n(s)\cdot 1)\leq d(y_n, \phi_n(s)\cdot y_n)\rbrace)=1$. Therefore, $d_{\omega}(y,s\cdot y)\geq d_{\omega}(x,s\cdot x)=1$. This is a contradiction since $y$ is supposed to be fixed by $G$.

The group $H$ fixes a point of $T$. Indeed, let $\lbrace h_1,\ldots ,h_p\rbrace$ be a finite generated set for $H$. Since the sequence $(\phi_n(h_i))_n$ is constant up to conjugation, the translation length $\ell_n(\phi_n(h_i))$ of $\phi_n(h_i)$ acting on $(X,d_n)$ goes to $0$ as $n$ goes to infinity. Hence, $h_i$ acts on the limit tree $T$ with translation length equal to 0. If follows that $h_i$ fixes a point of $T$, for every $1\leq i\leq p$. Similarly, $h_ih_j$ fixes a point of $T$ for every $1\leq i,j\leq p$. If follows from Serre's lemma that $H$ fixes a point of $T$.

The group $G$ acts hyperbolically on $T$, so there exists a (unique) minimal $G$-invariant subtree of $T$, namely the union of axes of hyperbolic elements of $G$. Hence, one can assume that the action of $G$ on $T$ is minimal and non-trivial. By Theorem 1.16 of \cite{RW14}, the action of $G$ on the limit tree $T$ has the following properties:
\begin{enumerate}
\item the stabilizer of any non-degenerate tripod is finite;
\item the stabilizer of any non-degenerate arc is virtually cyclic with infinite center;
\item the stabilizer of any unstable arc is finite.
\end{enumerate}
In particular, the tree $T$ satisfies the ascending chain condition of Theorem \ref{guirardel} since the stabilizer of any arc is virtually cyclic, and any ascending sequence of virtually cyclic subgroups of a hyperbolic group stabilizes.

Then, it follows from Theorem \ref{guirardel} that either $(G,H)$ splits over the stabilizer of an unstable arc, or over the stabilizer of an infinite tripod, or $T$ has a decomposition into a graph of actions. Since $G$ is one-ended relative to $H$, and since the stabilizer of an unstable arc or of an infinite tripod is finite, it follows that $T$ has a decomposition into a graph of actions.

Now, it follows from Theorems \ref{shorten1} and \ref{shorten2} (see also Remarks \ref{reqimportante} and \ref{reqimportante2}) that there exists a sequence of automorphisms $(\sigma_n)_n\in\mathrm{Aut}_H(G)^{\mathbb{N}}$ such that $\phi_n\circ \sigma_n$ is shorter than $\phi_n$ for $n$ large enough. This is a contradiction since the $\phi_n$ are assumed to be short.\end{proof}

The following theorem is a (relative) finiteness result for monomorphisms between two hyperbolic groups.

\begin{te}\label{sa}Let $\Gamma$ and $G$ be hyperbolic groups, and let $H$ be a finitely generated subgroup of $G$. Suppose that $G$ is one-ended relative to $H$. Assume in addition that $G$ embeds into $\Gamma$, and let $\psi : G \hookrightarrow \Gamma$ be a monomorphism. Then there exists a finite set $\lbrace i_1,\ldots ,i_{\ell}\rbrace\subset \mathrm{Hom}(G,\Gamma)$ of monomorphisms that coincide with $\psi$ on $H$ such that, for every monomorphism $\phi : G\hookrightarrow\Gamma$ that coincides with $\psi$ on $H$, there exist an automorphism $\sigma\in\mathrm{Aut}(G)$ such that $\sigma_{\vert H}=\mathrm{id}_H$, an integer $1\leq \ell \leq k$ and an element $\gamma\in \Gamma$ such that \[\phi=\mathrm{ad}(\gamma)\circ i_{\ell} \circ \sigma. \]\end{te}

\begin{rque}Note that the element $\gamma$ in the previous theorem belongs to $C_{\Gamma}(\psi(H))$. Indeed, for every $h\in H$, $\phi(h)=\psi(h)=\mathrm{ad}(\gamma)(i_{\ell}(\sigma(h)))=\mathrm{ad}(\gamma)(i_{\ell}(h))=\mathrm{ad}(\gamma)(\psi(h))$.
\end{rque}

\begin{proof}
Assume towards a contradiction that there exists a sequence of monomorphisms $(\phi_n : G \hookrightarrow \Gamma)$ such that for every $n$, $\phi_n$ coincides with $\psi$ on $H$, and for every $n\neq m$, for every $\gamma \in \Gamma$, for every $\sigma\in \mathrm{Aut}(G)$ whose restriction to $H$ is the identity, $\phi_n\neq \mathrm{ad}(\gamma)\circ \phi_m \circ \sigma$. Then, for every $n\neq m$, for every $\gamma \in \Gamma$, for every $\sigma\in \mathrm{Aut}_H(G)$, $\phi_n\neq \mathrm{ad}(\gamma)\circ \phi_m \circ \sigma$. Indeed, if $\sigma$ belongs to $\mathrm{Aut}_H(G)$, there exists an element $g\in G$ and an automorphism $\alpha$ of $G$ whose restriction to $H$ is the identity and such that $\sigma=\mathrm{ad}(g)\circ \alpha$, so $\mathrm{ad}(\gamma)\circ \phi_m\circ \sigma = \mathrm{ad}(\gamma\phi_m(g))\circ \phi_m\circ \alpha$. Hence, one can assume that the $\phi_n$ are short, pairwise distinct and injective. But the stable kernel of the sequence $(\phi_n)_n$ is trivial since each $\phi_n$ is injective. This contradicts Proposition \ref{propo}.\end{proof}

\begin{co}\label{coro}Let $G$ be a hyperbolic group, let $H$ be a finitely generated subgroup of $G$. Assume that $H$ is infinite, and that $G$ is one-ended relative to $H$. Then every monomorphism of $G$ whose restriction to $H$ is the identity is an automorphism of $G$.\end{co}

\begin{proof}Let $\phi$ be a monomorphism of $G$ whose restriction to $H$ is the identity. 

First, assume that the subgroup $H$ is non-elementary. Then $C_G(H)$ is finite. Since $\phi^n$ is a monomorphism for every integer $n$, it follows from Theorem \ref{sa} that there exist two integers $n>m$ such that $\phi^n=\mathrm{ad}(g)\circ \phi^m\circ \sigma$ for some $g\in C_G(H)$ and $\sigma\in \mathrm{Aut}(G)$ whose restriction to $H$ is the identity. We have $\phi(C_G(H))\subset C_G(H)$, so $\phi$ induces a bijection of $C_G(H)$. As a consequence, there exists an element $k\in C_G(H)$ such that $g=\phi^m(k)$, so $\phi^n=\phi^m\circ\mathrm{ad}(k)\circ \sigma$. Hence $\phi^{n-m}=\mathrm{ad}(k)\circ \sigma$ is an automorphism, so $\phi$ is an automorphism.

Now, assume that $H$ is infinite virtually cyclic. It fixes a pair of points $\lbrace x^{-}, x^{+}\rbrace$ on the boundary of $G$. The stabilizer $S$ of the pair of points $\lbrace x^{-},x^{+}\rbrace$ is virtually cyclic and contains $H$ and $C_G(H)$, which is infinite. Note that $H$ and $C_G(H)$ have finite index in $S$, so $Z(H)=H\cap C_G(H)$ has finite index as well. In particular, this group has finite index in $C_G(H)$. Let \[p=[C_G(H):Z(H)] \ \ \  \text{and} \ \ \ C_G(H)=\bigcup_{1\leq i\leq p}u_iZ(H).\]By Theorem \ref{sa}, there exists an increasing sequence of integers $(\ell_n)$ such that, for every $n$, there exist $
g_n\in C_G(H)$ and an automorphism $\sigma_n$ of $G$ whose restriction to $H$ is the identity, such that $\phi^{\ell_{n+1}}=\mathrm{ad}(g_n)\circ \phi^{\ell_n}\circ\sigma_n$. Up to taking a subsequence, one can assume that there exists an integer $i\in\llbracket 1,p\rrbracket$ such that $\ell_n$ belongs to $u_iZ(H)$ for each $n$. We have \[\phi^{\ell_{p}}=\mathrm{ad}(g_{p-1}\cdots g_0)\circ \phi^{\ell_0}\circ \sigma_0\circ\cdots \circ \sigma_{p-1}, \]with $g_{p-1}\cdots g_0$ belonging to $(u_iZ(H))^p=Z(H)$. Since $\phi$ coincides with the identity on $H$, we have $g_{p-1}\cdots g_0=\phi^{\ell_0}(g_{p-1}\cdots g_0)$, so \[\phi^{\ell_{p}-\ell_0}=\mathrm{ad}(g_{p-1}\cdots g_0)\circ \sigma_0\circ\cdots \circ \sigma_{p-1}.\]Hence $\phi$ is an automorphism.\end{proof}

Zlil Sela proved in \cite{Sel97} that torsion-free one-ended hyperbolic groups are co-Hopfian. Later, Christophe Moioli generalized this result in the presence of torsion in his PhD thesis (see \cite{Moi13}). By combining this result with Corollary \ref{coro}, we get the following theorem (note that a finitely generated group $G$ is one-ended relative to a finite subgroup $H$ if and only if $G$ is one-ended). 

\begin{te}\label{coHopf}Let $G$ be a hyperbolic group, let $H$ be a finitely generated subgroup of $G$. Assume that $G$ is one-ended relative to $H$. Then every monomorphism of $G$ whose restriction to $H$ is the identity is an automorphism of $G$.\end{te}

We conclude this section by proving the following finiteness result for non-injective homomorphisms between two hyperbolic groups, relative to a subgroup.

\begin{te}\label{short}Let $\Gamma$ and $G$ be hyperbolic groups, and let $H$ be a finitely generated subgroup of $G$ that embeds into $\Gamma$. Let $\psi : H \hookrightarrow \Gamma$ be a monomorphism. Assume that $G$ is one-ended relative to $H$. Then there exists a finite subset $F\subset G\setminus \lbrace 1\rbrace$ such that, for any non-injective homomorphism $\phi : G \rightarrow \Gamma$ that coincides with $\psi$ on $H$ up to conjugation, there exists a modular automorphism $\sigma\in\mathrm{Aut}_H(G)$ such that $\ker(\phi\circ\sigma)\cap F\neq \varnothing$.\end{te}

\begin{rque}Since $\mathrm{Mod}_H(G)$ has finite index in $\mathrm{Aut}_H(G)$ (see Theorem \ref{indice fini}), the result above is still valid if one replaces $\mathrm{Aut}_H(G)$ by $\mathrm{Mod}_H(G)$.\end{rque}

\begin{proof}
Let $B_n$ denote the set of elements of $G$ whose length is less that $n$, for a given finite generating set of $G$. Assume towards a contradiction that there exists a sequence $(\phi_n)$ of homomorphisms from $G$ to $\Gamma$ such that, for every integer $n$, the three following conditions hold:
\begin{enumerate}
\item $\phi_n$ coincides with $\psi$ on $H$ up to conjugation;
\item $\phi_n$ is non-injective;
\item for each $\sigma\in \mathrm{Aut}_H(G)$, $\phi_n\circ \sigma$ is injective in restriction to $B_n$.
\end{enumerate}
Up to precomposing each $\phi_n$ by an element of  $\mathrm{Aut}_H(G)$, and postcomposing by an inner automorphism, one can assume that $\phi_n$ is short. Moreover, up to passing to a subsequence, one can assume that the $\phi_n$ are pairwise distinct. But the stable kernel of the sequence $(\phi_n)_n$ is trivial, contradicting Proposition \ref{propo}.\end{proof}

\section{An example: the group $\mathrm{SL}_2(\mathbb{Z})$}\label{motivation}

In this section, we shall prove that the group $G=\mathrm{SL}_2(\mathbb{Z})$ is $\exists$-homogeneous. Recall that $G$ splits as $G=A\ast_C B$ with $A\simeq \mathbb{Z}/4\mathbb{Z}$, $B\simeq\mathbb{Z}/6\mathbb{Z}$ and $C\simeq\mathbb{Z}/2\mathbb{Z}$. Let $u$ and $v$ be two tuples in $G^k$ (for some $k\geq 1$) with the same (existential) type. We shall prove that there exists an automorphism of $G$ sending $u$ to $v$.

We shall decompose the proof into two steps. First, we shall prove that there exist a monomorphism $\phi$ of $G$ such that $\phi(u)=v$, and a monomorphism $\psi$ of $G$ such that $\phi(v)=u$. Then, we shall establish the existence of an automorphism of $G$ sending $u$ to $v$.

\subsection{First step}We claim that there exists a monomorphism $\phi$ of $G$ such that $\phi(u)=v$. Let $a$ denote a generator of $A$, and let $b$ denote a generator of $B$. A finite presentation of $G$ is given by $\langle a,b \ \vert \ \Sigma(a,b)=1\rangle$, where $\Sigma(a,b):(a^4=1)\wedge (b^6=1)\wedge (a^2=b^3)$. Let $u=(u_1,\ldots ,u_k)$ and $v=(v_1,\ldots ,v_k)$ with $u_i,v_i\in G$ for every $1\leq i\leq k$. Each element $u_i$ can be written as a word $w_i(a,b)$. The group $G$ satisfies the following sentence $\theta(u)$: \[\theta(u):\exists x\exists y \ \Sigma(x,y)=1  \ \bigwedge_{1\leq i\leq k}u_i=w_i(x,y)  \bigwedge_{1\leq \ell\leq 3} x^{\ell}\neq 1 \ \bigwedge_{1\leq \ell\leq 5} \ y^{\ell}\neq 1 .\]

Indeed, taking $x=a$ and $y=b$, the statement is obvious. Then, since $u$ and $v$ have the same existential type, the sentence $\theta(v)$ is satisfied by $G$ as well. This sentence asserts the existence of a tuple $(x,y)\in G^2$ which is solution to the system of equations and inequations \[\Sigma(x,y)=1  \ \bigwedge_{1\leq i\leq k}v_i=w_i(x,y)  \bigwedge_{1\leq \ell\leq 3} x^{\ell}\neq 1 \ \bigwedge_{1\leq \ell\leq 5} \ y^{\ell}\neq 1.\]In particular, $x$ has order 4 and $y$ has order 6. Since $\Sigma(x,y)=1$, the function $\phi : G \rightarrow G$ defined by $\phi(a)=x$ and $\phi(b)=y$ is an endomorphism of $G$. For every $1\leq i\leq k$, one has $v_i=w_i(x,y)=w_i(\phi(a),\phi(b))=\phi(w_i(a,b))=\phi(u_i)$. Hence, $\phi$ sends $u$ to $v$. It remains to prove that $\phi$ is injective.

First of all, note that $\phi$ is injective in restriction to $A$ and $B$. As a consequence, $\phi$ sends $A$ isomorphically to a conjugate $A^g$ of $A$, and $\phi$ sends $B$ isomorphically to a conjugate $B^h$ of $B$. Let $T$ be the Bass-Serre tree of the splitting $G=A\ast_C B$, let $v_A$ be a vertex fixed by $A$ and $v_B$ a vertex fixed by $B$. The group $C$ fixes the edge $[v_A,v_B]$, so $\phi(C)$ fixes the path $[g\cdot v_A,h\cdot w_B]$. But $C$ is normal in $T$, so it is the stabilizer of any edge of $T$; in particular, $C$ is the stabilizer of the path $[g\cdot v_A,h\cdot w_B]$. Thus, $\phi(C)$ is contained in $C$. But $\phi(C)$ has order $2$, so $\phi(C)=C$. 

Up to composing $\phi$ by $\mathrm{ad}(g^{-1})$, one can assume that $g=1$. Hence, one has $\phi(A)=A$ and $\phi(B)=B^h$. If $h$ belongs to $A$ or $B$, then $\phi$ is clearly surjective, so it is an automorphism of $G$ since $G$ is Hopfian. If $h$ does not belong to $A\cup B$, then one easily checks that $\phi$ is injective.

We have proved that there exists a monomorphism $\phi : G\hookrightarrow G$ which sends $u$ to $v$. Likewise, there exists a monomorphism $\psi : G\hookrightarrow G$ sending $v$ to $u$.

\subsection{Second step}We shall modify the monomorphism $\phi$ (if necessary) in order to get an automorphism of $G$ sending $u$ to $v$. First of all, let us prove the following lemma.

\begin{lemme}The group $G$ splits non-trivially over a finite subgroup relative to $\langle u\rangle$ if and only if $\langle u\rangle$ has finite order.
\end{lemme}

\begin{proof}Let $T$ denote the Bass-Serre tree associated with the splitting $A\ast_C B$ of $G$. If $\langle u\rangle$ has finite order, then $T$ gives a splitting of $G$ relative to $\langle u\rangle$.

Conversely, assume that $G$ splits non-trivially over a finite subgroup relative to $\langle u\rangle$. First, observe that $G$ does not admit any epimorphism onto $\mathbb{Z}$ since its abelianization is finite, so $G$ cannot split as an HNN extension. Consequently, $G$ splits as $U\ast_W V$ with $W$ finite and strictly contained in $U$ and $V$, and $\langle u\rangle\leq U$. Let $\Gamma$ be the Bass-Serre tree of this splitting. Let $T_U$ and $T_V$ be two reduced Stallings trees of $U$ and $V$ respectively, and let $T'$ be the Stallings tree of $G$ obtained by refining $\Gamma$ with $T_U$ and $T_V$. Up to forgetting the possibly redundant vertices of $T'$, one can assume that $T'$ is non-redundant. Since $T$ is the unique reduced Stallings tree of $G$, the tree $T'$ collapses onto $T$. Moreover, every expansion of $T$ is redundant. Thus, $T'=T$. It follows that $U$ is conjugate to $A$ or $B$. In particular, $\langle u\rangle$ has finite order.\end{proof}

We shall now prove that there exists an automorphism of $G$ which sends $u$ to $v$. There are two cases. 

\emph{First case.} If $G$ is one-ended relative to $\langle u\rangle$ (i.e.\ if $\langle u\rangle$ has infinite order), then $G$ is co-hopfian relative to $\langle u\rangle$ (cf. Theorem \ref{coHopf}). Therefore, the monomorphism $\psi\circ \phi$ is an automorphism of $G$, since it fixes $\langle u\rangle$. As a consequence, $\phi$ and $\psi$ are two automorphisms.

\emph{Second case.} If $G$ is not one-ended relative to $\langle u\rangle$ (i.e.\ if $\langle u\rangle$ has finite order), then $\phi$ is not an automorphism \textit{a priori}. We claim that it is possible to modify $\phi$ in such a way as to obtain an automorphism of $G$ which maps $u$ to $v$. First, observe that there exist two elements $g$ and $h$ of $G$ such that $\phi(A)=A^g$ and $\phi(B)=B^h$, because $\phi(A)$ has order 4 and $\phi(B)$ has order 6. Since $g$ and $h$ centralize $C$, one can define an endomorphism $\alpha$ of $G$ by setting $\alpha_{\vert A}=\mathrm{ad}(g^{-1})\circ\phi_{\vert A}$ and $\alpha_{\vert B}=\mathrm{ad}(h^{-1})\circ \phi_{\vert B}$. This homomorphism is surjective since its image contains $A$ and $B$. As $G$ is Hopfian, $\alpha$ is an automorphism of $G$. Moreover, since $v=\phi(u)$ has finite order, $\alpha$ sends $v$ to a conjugate $v^{\gamma}$ of $v$. Hence, the automorphism $\mathrm{ad}(\gamma^{-1})\circ \alpha$ sends $u$ to $v$. 

\begin{rque}Note that $\mathrm{SL}_n(\mathbb{Z})$ is $\exists$-homogeneous for $n\geq 3$ as well. Indeed, it follows from superrigidity that any endomorphism of $\mathrm{SL}_n(\mathbb{Z})$ (with $n\geq 3$) is either has finite image or is injective.\end{rque}

\subsection{Generalization}

In the next section \ref{mono}, we shall prove that the first step in the previous proof remains valid for any virtually free group $G$. More precisely, if $u$ and $v$ have the same type, then there is an endomorphism of $G$ which sends $u$ to $v$ and whose restriction to the maximal one-ended subgroup containing $\langle u\rangle$ is injective (see Proposition \ref{prop1}). As above, the proof will rely on first-order logic, but it will require $\forall\exists$-sentences in general (instead of $\exists$-sentences).

However, the second step does not work anymore in general: in Section \ref{contre-exemple}, we give a counterexample. 

\section{Isomorphism between relative one-ended factors}\label{mono}

The proposition below should be compared with the first step in the previous section.

\begin{prop}\label{prop1}
Let $G$ be a virtually free group, let $k\geq1$ be an integer and let $u,v\in G^k$. Let $U$ be the maximal one-ended subgroup of $G$ relative to $\langle u\rangle$, and let $V$ be the maximal one-ended subgroup of $G$ relative to $\langle v\rangle$. If $u$ and $v$ have the same type, then
\begin{itemize}
\item[$\bullet$]there exists an endomorphism $\phi$ of $G$ which sends $u$ to $v$ and whose restriction to $U$ is injective,
\item[$\bullet$]there exists an endomorphism $\psi$ of $G$ which sends $v$ to $u$ and whose restriction to $V$ is injective.
\end{itemize}
In fact, it is enough to suppose that $u$ and $v$ have the same $\forall\exists$-type.
\end{prop}

The following corollary is easy.

\begin{co}\label{corol}
Let $G$ be a virtually free group, let $k$ be an integer and let $u,v\in G^k$. Let $U$ be the maximal one-ended subgroup of $G$ relative to $\langle u\rangle$, and let $V$ be the maximal one-ended subgroup of $G$ relative to $\langle v\rangle$. If $u$ and $v$ have same type, then there exists an endomorphism of $G$ which maps $u$ to $v$ and induces an automorphism between $U$ and $V$.
\end{co}

\begin{proof}
Since $\phi(u)=v$, we have $\phi(U)\subset V$. Likewise, $\psi(V)\subset U$. The monomorphism $\psi\circ \phi_{\vert U} : U \hookrightarrow U $ fixes $\langle u\rangle$, thus it is an automorphism since $U$ is co-Hopfian relative to $\langle u\rangle$ (see Theorem \ref{coHopf}). Hence, $\phi$ induces an isomorphism between $U$ and $V$.\end{proof}

Before proving Proposition \ref{prop1}, we need some preliminary results. For now, let us admit the following lemma that will be usefull in the sequel (we refer the reader to Section \ref{preuvecyclic2} for the proof of this result).

\begin{lemme}\label{cyclic2}Let $G$ be a virtually free group, and let $\Delta$ be a centered splitting of $G$. Then $G$ does not admit any non-degenerate $\Delta$-preretraction.\end{lemme}

We shall also need the following easy result.

\begin{lemme}\label{petitlemme}
Let $G$ be a group with a splitting over finite groups. Denote by $T$ the associated Bass-Serre tree. Let $H$ be a group with a splitting over infinite groups, and let $S$ be the associated Bass-Serre tree. If $p:H\rightarrow G$ is a homomorphism injective on edge groups of $S$, and such that $p(H_v)$ is elliptic in $T$ for every vertex $v$ of $S$, then $p(H)$ is elliptic in $T$.
\end{lemme}

\begin{proof}Consider two adjacent vertices $v$ and $w$ in $S$. Let $H_v$ and $H_w$ be their stabilizers. The group $H_v\cap H_w$ is infinite by hypothesis. Moreover, $p$ is injective on edge groups, thus $p(H_v\cap H_w)$ is infinite. Hence $p(H_v)\cap p(H_w)$ is infinite. Since edge groups of $T$ are finite, $p(H_v)$ and $p(H_w)$ fix necessarily the same unique vertex $x$ of $T$. As a consequence, for each vertex $v$ of $S$, $p(H_v)$ fixes $x$. It follows that there exists a $p$-equivariant map $f : S \rightarrow T$ such that $f(S)=\lbrace x\rbrace$. Let $h$ be an element of $H$, and let $v$ be any vertex of $S$. One has $x=f(h\cdot v)=p(h)\cdot f(v)=p(h)\cdot x$, so $p(h)$ fixes the vertex $x$. Thus, $p(H)$ fixes $x$.\end{proof}

We shall now prove Proposition \ref{prop1}.

\begin{proof}
Assume towards a contradiction that the proposition is false. We shall build a centered splitting $\Delta$ of $G$, together with a non-degenerate $\Delta$-preretraction, and this will be a contradiction, according to Lemma \ref{cyclic2} above.

Here are the differents steps of the proof. 

\begin{itemize}
\item[$\bullet$]Step 1. Let $\Lambda$ denote the $\mathcal{Z}$-JSJ splitting of $U$ relative to $u$. We build a non-injective preretraction $p : U \rightarrow G$ (see Definition \ref{pre}).
\item[$\bullet$]Step 2. We prove that there is a QH vertex in $\Lambda$, denoted by $x$, such that $p$ does not send $U_x$ isomorphically to a conjugate of itself.
\item[$\bullet$]Step 3. We build a centered splitting $\Delta$ of $G$, whose central vertex is the vertex $x$ previously found. Then, we define an endomorphism $\phi$ of $G$ as follows: $\phi$ coincides with $p$ on the central vertex group $G_x=U_x$, and $\phi$ coincides with the identity map anywhere else. This endomorphism is a non-degenerate $\Delta$-preretraction.
\end{itemize}

Let us now move to the proofs.

\emph{Step 1.} We shall build a non-injective preretraction $p : U \rightarrow G$. First, let $F=\lbrace w_1,\ldots ,w_{\ell}\rbrace\subset U\setminus \lbrace 1\rbrace$ be the finite set given by Theorem \ref{short}, such that every non-injective homomorphism from $U$ to $G$ which coincides with the identity map on $\langle u\rangle$ kills an element of $F$, up to precomposition by a modular automorphism of $U$ relative to $\langle u\rangle$. Since we have assumed that Proposition \ref{prop1} is false, up to interchanging $u$ and $v$, every endomorphism $\phi$ of $G$ that sends $u$ to $v$ is non-injective in restriction to $U$. As a consequence, for every endomorphism $\phi$ of $G$ that sends $u$ to $v$, the homomorphism $\phi_{\vert U}$ is related (see Definition \ref{reliés2}) to a homomorphism $\psi : U\rightarrow G$ that kills an element of $F$ (this homomorphism $\psi$ is of the form $\phi_{\vert U}\circ\sigma$ for some $\sigma\in \mathrm{Mod}_{\langle u\rangle}(U)$). We shall see that this statement can be expressed by means of a first-order formula $\mu(\mathbf{z})$ with $k$ free variables satisfied by $v$. 

Let $G=\langle s_1,\ldots ,s_n \ \vert \ \Sigma(s_1,\ldots,s_n)\rangle$ be a finite presentation of $G$. Observe that there is a one-to-one correspondance between the set of endomorphisms of $G$ and the set of solutions in $G^n$ of the system of equations $\Sigma(x_1,\ldots,x_n)=1$. Let $U=\langle t_1,\ldots ,t_p \ \vert \ \Pi(t_1,\ldots ,t_p)\rangle$ be a finite presentation of $U$. Similarly, there is a one-to-one correspondance between $\mathrm{Hom}(U,G)$ and the set of solutions in $G^p$ of the system of equations $\Pi(x_1,\ldots,x_p)=1$. According to Lemma \ref{deltarelies}, there exists an existential formula $\theta(x_1,\ldots , x_{2p})$ with $2p$ free variables such that $\phi,\phi'\in \mathrm{Hom}(U,G)$ are $\Delta$-related if and only if the statement $\theta\left(\phi(t_1),\ldots , \phi(t_p),\phi'(t_1),\ldots ,\phi'(t_p)\right)$ is true in $G$. Each element $u_i$ can be seen as a word $u_i(s_1,\ldots,s_n)$ on the generators $s_1,\ldots ,s_n$ of $G$. Likewise, each generator $t_i$ of $U< G$ can be seen as word $t_i(s_1,\ldots ,s_n)$ on the generators of $G$, and each element $w_i\in F$ can be seen as a word $w_i(t_1,\ldots ,t_p)$ on the generators of $U$. We define the formula $\mu(\mathbf{z})$ as follows:
\begin{align*}
\mu(\mathbf{z}) & : \forall x_1 \ldots\forall x_n \  ((\Sigma(x_1,\ldots,x_n)=1)\wedge \underset{i\in\llbracket 1,k\rrbracket}{\bigwedge}z_i=u_i(x_1,\ldots ,x_n)) \\
 & \Rightarrow \exists y_1, \ldots \exists y_p \left(\rule{0cm}{1cm}
       \begin{array}{lc}
                        &  \Pi(y_1,\ldots ,y_p)=1 \\
        & \wedge \ \underset{i\in\llbracket 1,\ell\rrbracket}{\bigvee} w_i(y_1,\ldots,y_p)=1 \\
                        &   \wedge \ \theta(t_1(x_1,\ldots ,x_n),\ldots ,t_p(x_1,\ldots ,x_n),y_1,\ldots ,y_p) \\                                       
        \end{array}
     \right).
\end{align*}
The statement $\mu(v)$ is true in $G$. Indeed, taking $x_i=s_i$ for every $1\leq i\leq n$, this statement tells us that for every endomorphism $\phi$ of $G$ which maps $u$ to $v$ (defined by setting $\phi(s_i)=x_i$ for every $1\leq i\leq n$), there exists a homomorphism $\psi : U \rightarrow G$, defined by setting $\psi(t_i)=y_i$ for every $1\leq i\leq p$, such that $\ker(\psi)\cap F$ is non-empty and $\psi$ and $\phi_{\vert U}$ are $\Delta$-related. Since $u$ and $v$ have the same type, the statement $\mu(u)$ is true in $G$ as well. In other words, for every endomorphism $\phi$ of $G$ that sends $u$ to itself, the homomorphism $\phi_{\vert U}$ is related to a homomorphism $\psi : U\rightarrow G$ that kills an element of $F$. Now, taking for $\phi$ the identity of $G$, we get a non-injective preretraction $p :U\rightarrow G$.

Note that $\mu(\mathbf{z})$ is a $\forall\exists$-formula. Consequently, it is enough to assume that $u$ and $v$ have the same $\forall\exists$-type.

\emph{Step 2.} We shall prove that there exists a QH vertex $x$ of $\Lambda$ such that $U_x$ is not sent isomorphically to a conjugate of itself by $p$. Recall that $\Delta_u$ is a Stallings splittings of $G$ relative to $u$. Assume towards a contradiction that each stabilizer $U_x$ of a QH vertex $x$ of $\Lambda$ is sent isomorphically to a conjugate of itself by $p$. In particular, $p(U_x)$ is contained in a conjugate of $U$ in $G$. As a consequence, $p(U_x)$ is elliptic in the Bass-Serre tree $T_u$ of $\Delta_u$, for every QH vertex $x$. On the other hand, if $y$ is a non-QH vertex of $\Lambda$, $p(U_y)$ is elliptic in $T_u$ by definition of $\Lambda$-relatedness. Therefore, it follows from Lemma \ref{petitlemme} that $p(U)$ is elliptic in $T_u$, because $p$ is injective on edge groups of $\Lambda$, and $T_u$ has finite edge groups. Moreover, since $p$ is inner on non-QH vertices of $\Lambda$, the group $p(U)$ is contained in $gUg^{-1}$ for some $g\in G$ (note that there exists at least one non-QH vertex since $U$ is not finite-by-(closed orbifold), as a virtually free group). Up to composing $p$ by the conjugation by $g^{-1}$, one can thus assume that $p$ is an endomorphism of $U$. Now, by Proposition \ref{lemmeperin}, $p$ is injective. This is a contradiction. Hence, we have proved that there exists a QH vertex $x$ of $\Lambda$ such that $U_x$ is not sent isomorphically to a conjugate of itself by $p$.

\emph{Step 3.} It remains to construt a centered splitting of $G$ (see Definition \ref{graphecentre}). First, we refine $\Delta_u$ by replacing the vertex fixed by $U$ by the JSJ splitting $\Lambda$ of $U$ relative to $u$. With a little abuse of notation, we still denote by $x$ the vertex of $\Delta_u$ corresponding to the QH vertex $x$ of $\Lambda$ defined in the previous step. Then, we collapse to a point every connected component of the complement of $\mathrm{star}(x)$ in $\Delta_u$ (where $\mathrm{star}(x)$ stands for the subgraph of $\Delta_u$ constituted of $x$ and all its incident edges). The resulting graph of groups, denoted by $\Delta$, is non-trivial. One easily sees that $\Delta$ is a centered splitting of $G$, with central vertex $x$.

The homomorphism $p:U\rightarrow G$ is well-defined on $G_x$ because $G_x=U_x$ is contained in $U$. Moreover, $p$ restricts to a conjugation on each stabilizer of an edge $e$ of $\Delta$ incident to $x$. Indeed, either $e$ is an edge coming from $\Lambda$, either $G_e$ is a finite subgroup of $U$; in each case, $p_{\vert G_e}$ is a conjugation since $p$ is $\Lambda$-related to the inclusion of $U$ into $G$. Therefore, one can define an endomorphism $\phi : G \rightarrow G$ that coincides with $p$ on $G_x=U_x$ and coincides with a conjugation on every vertex group $G_y$ of $\Delta_u$, with $y\neq x$. Hence, the endomorphism $\phi$ is $\Delta$-related to the identity of $G$ (in the sense of Definition \ref{reliés}), and $\phi$ does not send $G_x$ isomorphically to a conjugate of itself, by Step 2. Hence, $\phi$ is a non-degenerate $\Delta$-preretraction (see Definition \ref{special}).\end{proof}

It remains to prove Lemma \ref{cyclic2}.

\subsection{Proof of Lemma \ref{cyclic2}}\label{preuvecyclic2}We need two preliminary results. 

\begin{lemme}\label{deuxieme}Let $G$ be a hyperbolic group, and let $\Delta$ be a centered splitting of $G$. If $G$ admits a non-degenerate $\Delta$-preretraction, then it has a subgroup $H$ with the following property: there exists a non-trivial minimal splitting $\Gamma$ of $H$ over virtually cyclic groups with infinite center such that, for every vertex $x$ of $\Gamma$, the vertex group $H_x$ does not split non-trivially over a finite group relative to the stabilizers of edges incident to $x$ in $\Gamma$.\end{lemme}

Recall that a splitting $\Gamma$ is said to be minimal if there is no proper subtree of the Bass-Serre tree $T$ of $\Gamma$ invariant under the action of $G$. Equivalently, $\Gamma$ is minimal if and only if $T$ has no vertex of degree equal to one, i.e.\ if $\Gamma$ has no vertex $v$ of degree equal to one such that $G_v=G_e$, where $e$ denotes the unique edge incident to $v$ in $\Gamma$.

\begin{proof}
Let $v$ denote the central vertex of $\Delta$. Let $\phi$ be a non-degenerate $\Delta$-preretraction. Let $\Delta'$ denote the splitting of $G$ obtained from $\Delta$ by replacing each vertex $w\neq v$ by a Stallings splitting $\Lambda_w$ of $G_w$ relative to the stabilizers of edges incident to $w$ in $\Delta$ (see Figure \ref{dessin} below). Let $\mathcal{F}$ be the set of edges $e$ of $\Delta'$ such that $G_e$ is finite. Let $\Gamma$ denote the connected component of $v$ in $\Delta'\setminus \mathcal{F}$ (see Figure \ref{dessin} below), and let $H$ be the fundamental group of $\Gamma$. 

\begin{figure}[!h]
\includegraphics[scale=0.7]{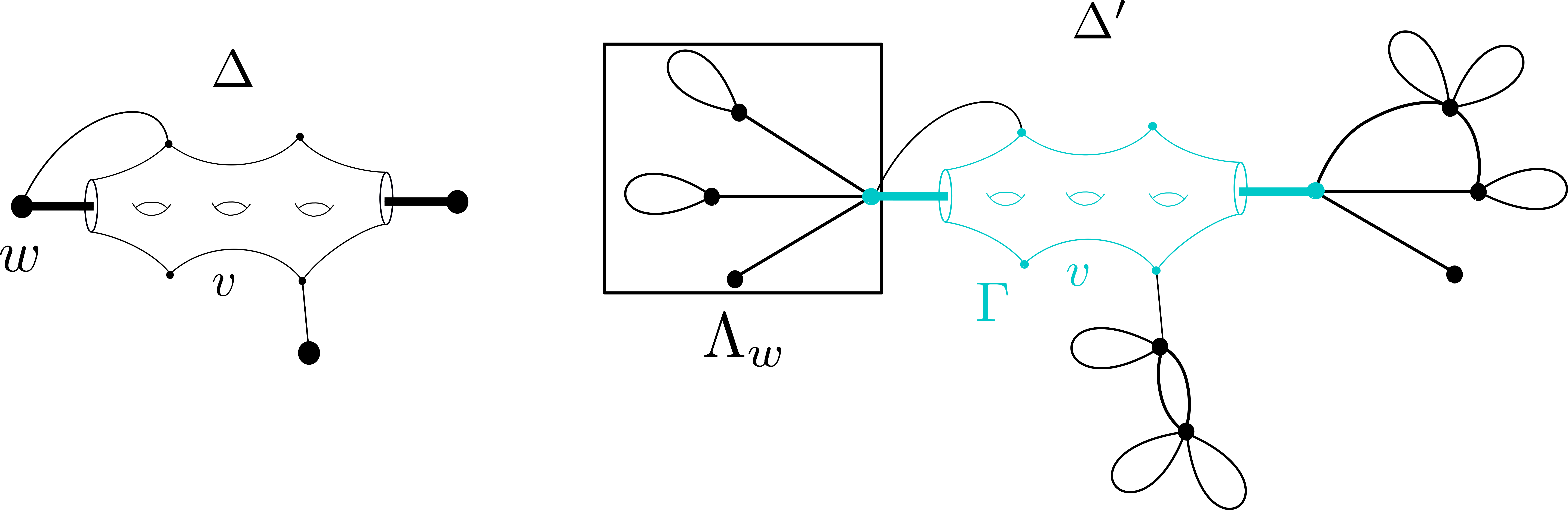}
\caption{The graphs of groups $\Delta$, $\Delta'$ and $\Gamma$. The latter is pictured in blue as a subgraph of $\Delta'$. Edges with infinite stabilizer are depicted in bold.}
\label{dessin}
\end{figure}

Note that for every vertex $x\neq v$ of $\Gamma$, the vertex group $H_x$ is one-ended relative to the incident edge groups. Moreover, since $H_v=G_v$ is an orbifold group, it does not split non-trivially over a finite group relative to its boundary subgroups. In order to prove the lemma, it remains to prove that $\Gamma$ is minimal. It is enough to prove that for every vertex $x\neq v$ of $\Gamma$ such that there is exactly one edge $e$ between $x$ and $v$ in $\Gamma$, the stabilizer $H_e$ of the edge $e$ is strictly contained in $H_x$.

Let $x$ be such a vertex of $\Gamma$, and let $e$ denote the unique edge between $x$ and $v$. By construction, there exists a vertex $w$ of $\Delta$ such that $x$ is a vertex of the Stallings splitting $\Lambda_w$ of $G_w$ relative to the stabilizers of edges incident to $w$ in $\Delta$. 

Let $S$ be a maximal pinched set for $\phi$ and let $\Delta(S)$ be the splitting of $G$ obtained from $\Delta$ by replacing the central vertex $v$ by the splitting of $G_v$ dual to $S$. Let $v_1,\ldots ,v_n$ be the new vertices coming from $v$ (see Figure \ref{dessin2} below).

\begin{figure}[!h]
\includegraphics[scale=0.7]{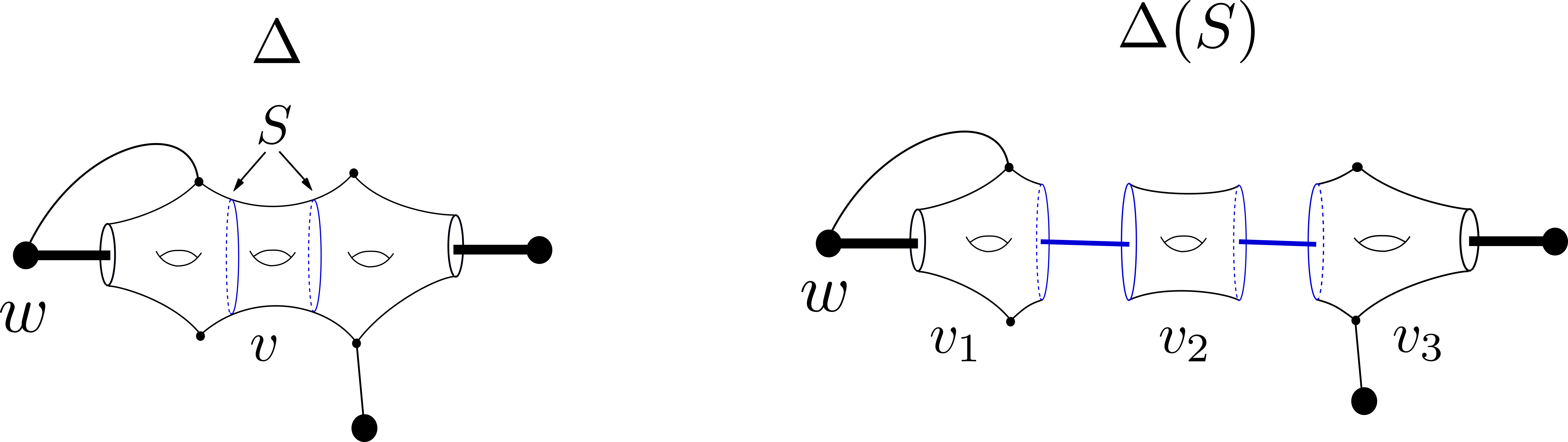}
\caption{On the left, the centered splitting $\Delta$ of $G$; edges with infinite stabilizer are depicted in bold. On the right, the splitting $\Delta(S)$ of $G$ is obtained from $\Delta$ by replacing the central vertex $v$ by the splitting of $G_v$ dual to $S$.}
\label{dessin2}
\end{figure}

We denote by $\mathcal{S}$ the set of edges of $\Delta(S)$ corresponding to the maximal pinched set $S$, and we denote by $\mathcal{F}$ the set of edges of $\Delta(S)$ whose stabilizer is finite. Let $W$ be the connected component of $\Delta(S)\setminus (\mathcal{S}\cup \mathcal{F})$ that contains the vertex $w$, and let $G_W$ denote its stabilizer (see Figure \ref{dessin3} below). By Lemma \ref{magique}, $W$ contains exactly two vertices: the vertex $w$ and a vertex $v_k$ coming from $v$. Moreover, $\phi(G_W)=\phi(G_w)$. 

\begin{figure}[!h]
\includegraphics[scale=0.7]{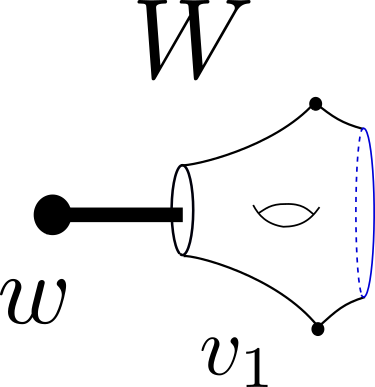}
\caption{$W$ is the connected component of $\Delta(S)\setminus (\mathcal{S}\cup \mathcal{F})$ that contains the vertex $w$.}
\label{dessin3}
\end{figure}

By definition of a non-degenerate $\Delta$-preretraction, there exists an element $g\in G$ such that the restriction of $\phi$ to $G_w$ coincides with the inner automorphism $\mathrm{ad}(g)$. Consequently, up to composing $\phi$ with $\mathrm{ad}(g^{-1})$, we can assume that $\phi$ coincides with the identity map on $G_w$ (i.e.\ that $\phi_{\vert G_W} : G_W \twoheadrightarrow G_w$ is a retraction). 

Let us refine the graph of groups $W$ by replacing the vertex $w$ by the Stallings splitting $\Lambda_w$ of $G_w$ relative to the stabilizers of edges incident to $w$. Let $\Lambda_W$ denote this new graph of groups (see Figure \ref{dessin4} below). 

\begin{figure}[!h]
\includegraphics[scale=0.7]{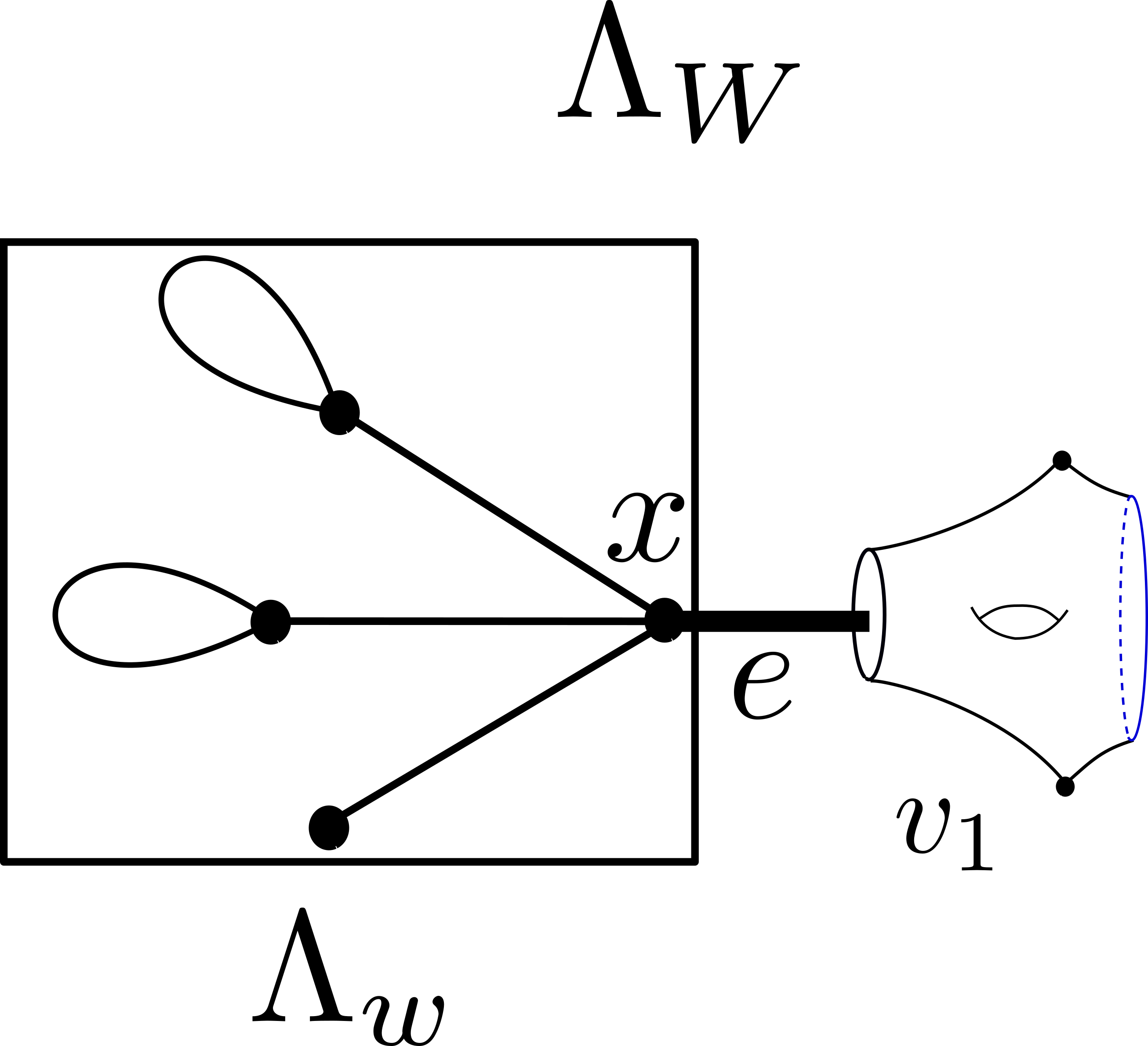}
\caption{$\Lambda_W$ is the splitting of $G_W$ obtained from $W$ by replacing $w$ by the splitting $\Lambda_w$ of $G_w$. The endomorphism $\phi$ of $G$ induces a retraction $\phi_{\vert G_W}$ from $G_W$ onto $G_w$.}
\label{dessin4}
\end{figure}

Assume towards a contradiction that $G_x=G_e$. Recall that the stabilizer $G_{v_k}$ of $v_k$ is an extension $F\hookrightarrow G_{v_k}\twoheadrightarrow \pi_1(\mathcal{O}_k)$ where $\mathcal{O}_k$ is a compact hyperbolic 2-orbifold, and $F$ is a finite group called the fiber. The image of $G_e$ in $\pi_1(\mathcal{O}_k)$ is an infinite cyclic group $\langle\gamma\rangle$. Let $\bar{\gamma}$ be a preimage of $\gamma$ in $G_{v_k}$. One has $G_x=G_e=F\rtimes\langle \bar{\gamma}\rangle$. We claim that $G_w$ surjects onto $\langle \bar{\gamma}\rangle$. For every edge $\varepsilon$ incident to $x$ in $\Lambda_w$, the stabilizer $G_{\varepsilon}$ is a subgroup of $F$. Let $V(\Lambda_w)$ denote the set of vertices of $\Lambda_w$, and let $N$ be the subgroup of $G_w$ normally generated by $\lbrace G_y \ \vert \ y\in V(\Lambda_w), \ y\neq x\rbrace\cup F$. The quotient group $G_w/N$ surjects onto $\langle\bar{\gamma}\rangle$, so $G_w$ surjects onto $\langle \bar{\gamma}\rangle$. Let $r :G_w\twoheadrightarrow \langle\bar{\gamma}\rangle$ denote this epimorphism. Let $n$ denote the genus of the orbifold $\mathcal{O}_k$, let $c_1,\ldots ,c_p$ denote the conical elements, and let $\gamma_1,\ldots,\gamma_q$ denote the generators of the boundary subgroups of $\mathcal{O}_k$ different from $\gamma$. 

In the case where the orbifold $\mathcal{O}_k$ is orientable, there exist $2n$ elements $a_1,b_1,\ldots,a_n,b_n$ in $\pi_1(\mathcal{O}_k)$ such that the following relation holds in $\pi_1(\mathcal{O}_k)$:\[\gamma=[a_1,b_1]\cdots [a_n,b_n]c_1\cdots c_p \gamma_1\cdots \gamma_q.\]For every element $\alpha$ in $\pi_1(\mathcal{O}_k)$, let us denote by $\bar{\alpha}$ a preimage of $\alpha$ in $G_{v_k}$. There exists an element $z\in F$ such that the following relation holds in $G_{v_k}$:\[\bar{\gamma}=[\bar{a}_1,\bar{b}_1]\cdots [\bar{a}_n,\bar{b}_n]\bar{c}_1\cdots \bar{c}_p \bar{\gamma}_1\cdots \bar{\gamma}_qz.\]One has $r\circ\phi(z)=1$ and $r\circ \phi(\bar{c}_i)=1$ for $1\leq i\leq p$, because $z$ and $\bar{c_i}$ have finite order in $G_{v_k}$, and $\langle\bar{\gamma}\rangle$ is torsion-free. Likewise, $r\circ \phi(\bar{\gamma}_i)=1$, for $1\leq i\leq q$, because $\bar{\gamma}_i$ is pinched by $\phi$. Moreover, since $\langle\bar{\gamma}\rangle$ is abelian, the image of the commutator $[\bar{a}_i,\bar{b}_i]$ by $r\circ \phi$ is trivial, for every $1\leq i\leq n$. As a consequence, the element $r\circ\phi(\bar{\gamma})$ is trivial. But $r\circ\phi(\bar{\gamma})=\bar{\gamma}$ since $r\circ\phi :  G_W\twoheadrightarrow G_x=G_e=\langle\bar{\gamma}\rangle$ is a retraction. This is a contradiction. 

In the case where the orbifold $\mathcal{O}_k$ is non-orientable, there exist $n$ elements $a_1,\ldots ,a_n$ such that the following relation holds in $\pi_1(\mathcal{O}_k)$: \[\gamma=a_1^2\cdots a_n^2 c_1\cdots c_p \gamma_1\cdots \gamma_q.\]There exists an element $z\in F$ such that the following relation holds in $G_{v_k}$:\[\bar{\gamma}=\bar{a}_1^2\cdots \bar{a}_n^2\bar{c}_1\cdots \bar{c}_p \bar{\gamma}_1\cdots \bar{\gamma}_qz.\]It follows that $r\circ\phi(\bar{\gamma})=(r\circ\phi(\bar{a}_1\cdots \bar{a}_n))^2=\bar{\gamma}$. Let $r' : \bar{\gamma}\twoheadrightarrow \mathbb{Z}/2\mathbb{Z}$ be the epimorphism obtained by killing the squares. Since $r\circ\phi(\bar{\gamma})$ is a square, one has $r'\circ r\circ \phi(\bar{\gamma})=1$. But $r'\circ r\circ \phi(\bar{\gamma})$ has order two, since $r\circ \phi(\bar{\gamma})=\bar{\gamma}$. This is a contradiction.\end{proof}

\color{black}

Recall that every virtually cyclic group $G$ with infinite center splits as $F\hookrightarrow G \twoheadrightarrow \mathbb{Z}$, with $F$ finite. In particular, the group $G$ cannot be generated by two finite subgroups, since its abelianization $G^{\mathrm{ab}}$ maps onto $\mathbb{Z}$ as well. Note also that any infinite subgroup of $G$ is virtually cyclic with infinite center.
The following result is adapted from \cite{Hor17}, Lemma 5.11.

\begin{lemme}\label{penible}Let $G$ be a virtually free group with a non-trivial minimal splitting $\Gamma$ over virtually cyclic groups with infinite center. There exists a vertex $x$ of $\Gamma$ such that $G_x$ splits non-trivially over a finite group relative to the stabilizers of edges incident to $x$ in $\Gamma$.\end{lemme}

Before proving the lemma above, we need a definition. Let $T,T'$ be two simplicial $G$-trees. A map $f : T\rightarrow T'$ is a \emph{fold} if one of the two following situations occurs.

\begin{enumerate}
\item Either there exist two edges $e=[v,w]$ and $e'=[v,w']$ in $T$, incident to a common vertex $v$ in $T$ and belonging to different $G$-orbits, such that $T'$ is obtained from $T$ by $G$-equivariantly identifying $e$ and $e'$, and $f : T \rightarrow T'$ is the quotient map (see Figure below). The fold $f : T\rightarrow T'$ is determined by the orbit of the pair of edges $(e,e')$ identified by $f$.

\begin{center}
\begin{tikzpicture}[scale=1]
\node[draw,circle, inner sep=1.7pt, fill, label=below:{$w$}] (A1) at (2,0) {};
\node[draw,circle, inner sep=1.7pt, fill, label=below:{$w'$}] (A2) at (2,2) {};
\node[draw,circle, inner sep=1.7pt, fill, label=below:{$v$}] (A3) at (0,1) {};
\node[draw=none, label=below:{$e$}] (B1) at (1,0.5) {};
\node[draw=none, label=below:{$e'$}] (B2) at (1,2.2) {};
\node[draw=none, label=below:{$f(e)=f(e')$}] (B3) at (7,2) {};
\node[draw,circle, inner sep=1.7pt, fill, label=below:{$f(w)=f(w')$}] (A4) at (8,1) {};
\node[draw,circle, inner sep=1.7pt, fill, label=below:{$f(v)$}] (A5) at (6,1) {};

\draw[-,>=latex] (A3) to (A1) ;
\draw[-,>=latex] (A3) to (A2);
\draw[-,>=latex] (A4) to (A5);
\draw[->,>=latex, dashed] (3,1) to (5,1);
\end{tikzpicture}
\end{center}
There are two distinct subcases:
\begin{enumerate}
\item if $w$ and $w'$ belong to distinct $G$-orbits, then one has $G_{f(w)}=\langle G_w,G_{w'}\rangle$ and $G_{f(e)}=\langle G_e,G_{e'}\rangle$; 
\item if $w'=gw$ and $e'=ge$, then $G_{f(w)}=\langle G_w,g\rangle$ and $G_{f(e)}=\langle G_e,G_{e'}\rangle$.
\end{enumerate}
\item Or there exist an edge $e=[v,w]$ in $T$ and a subgroup $H$ of $G_v$, such that $T'$ is obtained from $T$ by $G$-equivariantly identifying $e$ and $he$ for every $h\in H$, and $f : T \rightarrow T'$ is the quotient map. We have $G_{f(w)}=\langle G_w,H\rangle$ and $G_{f(e)}=\langle G_e,H\rangle$. The fold $f : T\rightarrow T'$ is determined by the orbit of the pair $(e,H)$.

\begin{center}
\begin{tikzpicture}[scale=1]
\node[draw,circle, inner sep=1.7pt, fill, label=below:{$w$}] (A1) at (2,1) {};
\node[draw,circle, inner sep=1.7pt, fill] (A2) at (1.4141,1.4142+1) {};
\node[draw,circle, inner sep=1.7pt, fill] (A11) at (0,3) {};
\node[draw,circle, inner sep=1.7pt, fill] (A9) at (-1.4141,1.4142+1) {};
\node[draw,circle, inner sep=1.7pt, fill] (A10) at (-2,1) {};
\node[draw,circle, inner sep=1.7pt, fill] (A6) at (-1.4141,-1.4142+1) {};
\node[draw,circle, inner sep=1.7pt, fill] (A7) at (-2,1) {};
\node[draw,circle, inner sep=1.7pt, fill] (A8) at (-1.4141,-1.4142+1) {};
\node[draw,circle, inner sep=1.7pt, fill] (A12) at (0,-1) {};
\node[draw,circle, inner sep=1.7pt, fill] (A13) at (1.4141,-1.4142+1) {};
\node[draw,circle, inner sep=1.7pt, fill, label=below:{$v$}] (A3) at (0,1) {};
\node[draw=none, label=below:{$e$}] (B1) at (1,1.5) {};
\node[draw=none, label=below:{$he$}] (B2) at (0.7,2.6) {};
\node[draw=none, label=below:{$\langle G_e,H\rangle$}] (B3) at (7,2) {};
\node[draw,circle, inner sep=1.7pt, fill, label=below:{$\langle G_w,H\rangle$}] (A4) at (8,1) {};
\node[draw,circle, inner sep=1.7pt, fill, label=below:{$G_v$}] (A5) at (6,1) {};

\draw[-,>=latex] (A3) to (A1) ;
\draw[-,>=latex] (A3) to (A12) ;
\draw[-,>=latex] (A3) to (A13) ;
\draw[-,>=latex] (A3) to (A2);
\draw[-,>=latex] (A3) to (A9) ;
\draw[-,>=latex] (A3) to (A10) ;
\draw[-,>=latex] (A3) to (A11) ;
\draw[-,>=latex] (A3) to (A6) ;
\draw[-,>=latex] (A3) to (A7) ;
\draw[-,>=latex] (A3) to (A8) ;
\draw[-,>=latex] (A4) to (A5);
\draw[->,>=latex, dashed] (3,1) to (5,1);
\end{tikzpicture}
\end{center}
\end{enumerate}

\begin{rque}\label{...}Let $e=[v,w]$ and $e'=[v,w']$ be two edges of $T$ incident to a common vertex $v$. Let $f: T \rightarrow T'$ be a fold. Suppose that $G$ acts on the quotient $T'$ without inversion. If $e'=ge$, then $g$ fixes the vertex $v$. Hence, there exists a subgroup $H$ of $G_v$ containing $g$ such that $f$ is defined by the pair $(e,H)$.\end{rque}

\vspace{2mm}

\begin{proof2}
Let $T$ be a Stallings tree of $G$, and let $T'$ be the Bass-Serre tree of $\Gamma$. All vertex stabilizers of $T$ are finite, so elliptic in $T'$. As a consequence, there exists an equivariant map $f : T \rightarrow T'$ that sends each edge of $T$ to a point of $T'$ or to a path of edges in $T'$. Up to possibly collapsing some edges in $T$, and subdividing edges of $T$, one can suppose that $f$ maps each edge of $T$ to an edge of $T'$. 

We claim that there exists a finite sequence of $G$-trees $(T_k)_{0\leq k\leq n+1}$ with $n\geq 0$ and $T_0=T$, such that 
\begin{itemize}
\item[$\bullet$]$T_{k+1}$ is obtained from $T_k$ by a fold $f_k$,
\item[$\bullet$]$T_{n+1}$ is equivariantly isometric to $T'$ (we note $T_{n+1}=T'$), 
\item[$\bullet$]and the fold $f_n : T_n\rightarrow T'$ involves (at least) one edge with finite stabilizer. 
\end{itemize}

Before proving the claim, we explain how to derive the lemma from this claim. The fold $f_n$ identifies two distinct adjacent edges $e=[v,w]$ and $e'=[v,w']$ of $T_n$, with $G_e$ finite. We shall find a non-trivial splitting of the vertex group $G_{f_n(w)}$ over a finite group, relative to the stabilizers of edges incident to $f_n(w)$ in $T'$.

\begin{enumerate}
\item If $w$ and $w'$ are not in the same $G$-orbit, the stabilizer of $f_n(w)\in T'$ splits as $G_w\ast_{G_e\cap G_{e'}}G_{w'}=G_w\ast_{G_e}\langle G_{w'},G_e\rangle$. By the previous paragraph, this splitting is non-trivial. Moreover, the stabilizer $\langle G_e,G_{e'}\rangle$ of the edge $f_n(e)$ incident to $f_n(w)$ is contained in $\langle G_{w'},G_e\rangle$. In addition, if $\varepsilon$ is an edge of $T_n$ incident to $w$ or to $w'$, then the edge $f_n(\varepsilon)$ is incident to $f_n(w)$ in $T'$ and its stabilizer is contained in one a the two factors of $G_{f_n(w)}$.
\item If $w'=gw$ and $e'=ge$, then $G_{f_n(w)}$ splits as $\langle G_e,g\rangle\ast_{G_e} G_w$. This splitting is non-trivial and relative to the stabilizers of the edges incident to $f_n(w)$ in $T'$.
\item If $w'=gw$ but $e$ and $e'$ are not in the same $G$-orbit, then $G_{f_n(w)}$ splits as an HNN-extension $G_{w'}\ast_{G_e\cap G_{e'}}=\langle G_{w'},G_e\rangle\ast_{G_e}$. As above, this splitting is non-trivial and relative to the stabilizers of the edges incident to $f_n(w)$ in $T'$.
\end{enumerate}

In cases 1 and 2, it remains to prove that the splitting of $G_{f_n(w)}$ obtained above is non-trivial. It is enough to prove that $G_e$ is strictly contained in $G_w$. Assume towards a contradiction that $G_e=G_w$, and let us consider an edge $\varepsilon\neq e$ adjacent to $w$ in $T_n$. Since $G_w=G_e$ is finite, $G_{\varepsilon}$ is finite as well. But there are at most two orbits of edges with finite stabilizer in $T_n$ since $f_n$ is the last fold of the sequence, namely the orbit of $e$ and (possibly) the orbit of $e'$. If $\varepsilon=ge$, then either $g$ fixes $w$, so $g$ belongs to $G_w=G_e$, a contradiction; either $g$ is hyperbolic with translation length equal to 1, so $w=gv$. Thus, $G_v$ is finite. This is impossible because $G_v=G_{f_n(v)}$ is a vertex group of $T'$, and edge groups of $T'$ are infinite. If $\varepsilon=ge'$, then $g$ is hyperbolic with translation length equal to 2, so $w=gw'$. In particular, $G_{w'}$ and $G_{e'}$ are finite. By hypothesis, the edge group $G_{f_n(e)}=\langle G_e,G_{e'}\rangle$ is virtually cyclic with infinite center, so it maps onto $\mathbb{Z}$. But $G_e$ and $G_{e'}$ are finite, so $\langle G_e,G_{e'}\rangle$ has finite abelianization. This is a contradiction. Hence, $G_e$ is strictly contains in $G_w$. 

It remains to prove the claim. By \cite{Sta83}, we know that the map $f : T \rightarrow T'$ can be decomposed into a sequence of $G$-equivariant edge folds $(f_k : T_k\twoheadrightarrow T_{k+1})_{0\leq k\leq n}$, with $T_0=T$ and $T_{n+1}=T'$. But in general there is no reason why the last fold $f_n : T_n\rightarrow T'$ should involve an edge with finite stabilizer. We shall prove that it is always possible to find a sequence satisfying this condition, performing the folds in a certain order. Using the terminology of \cite{Hor17}, we say that a fold is 
\begin{itemize}
\item[$\bullet$]of type 1 if it identifies two edges $e$ and $e'$ belonging to distinct $G$-orbits, and both $e$ and $e'$ have infinite stabilizer, and 
\item[$\bullet$]of type 2 if it identifies two edges $e$ and $e'$ belonging to distinct $G$-orbits, and either $e$ or $e'$ (or both) has finite stabilizer, and
\item[$\bullet$]of type 3 if it identifies two edges belonging to the same $G$-orbit. 
\end{itemize}

Since the number of orbits of edges decreases when performing a fold of type 1 or 2, we can assume that along the folding sequence, we only perform a fold of type 2 if no fold of type 1 is possible, and we only perform a fold of type 3 if no fold of type 2 is possible. 

For every $k\in \llbracket 1,n\rrbracket$, let $\phi_k : T_k \rightarrow T'$ be the unique map such that $f=\phi_k\circ f_k\circ \cdots\circ f_0$. We shall construct by induction a sequence of folds with the following \emph{maximality property}: for every edge $e_k$ of $T_k$ with infinite stabilizer, the stabilizer of $\phi_k(e_k)\in T'$ is equal to $G_{e_k}$. In particular, if $f_k$ identifies two adjacent edges $e_k$ and $e'_k$ of $T_k$ with infinite stabilizers, then $e_k$ and $e'_k$ belong to distinct $G$-orbits and $G_{e_k}=G_{e'_k}$. 

The maximality property obviously holds for $T_0$, because all stabilizers of edges in $T_0$ are finite. Now, suppose that there exists a sequence $(T_i)_{0\leq i\leq k}$ satisfying the maximality property. If $T_k\neq T'$, there exists a fold $f_k : T_k \twoheadrightarrow T_{k+1}$. We can assume that $f_k$ is of type $t\in \lbrace 1,2,3\rbrace$ with $t$ as small as possible. Let $e_k$ and $e'_k$ be two adjacent edges of $T_k$ that are identified by $f_k$. Let $e=\phi_k(e_k)=\phi_k(e'_k)\in T'$. There are three distinct cases.
\begin{enumerate}
\item If $e_k$ or $e'_k$ has infinite stabilizer, then $G_{e_k}=G_{e}$ or $G_{e'_k}=G_{e}$ according to our induction assumption, so $G_{f_k(e_k)}=G_{e}$ in both cases. Hence, $T_{k+1}$ has the maximality property.
\item If $e_k$ and $e'_k$ have finite stabilizers and $G_{f_k(e_k)}$ is finite, then $T_{k+1}$ has the maximality property.
\item If $e_k$ and $e'_k$ have finite stabilizers and $G_{f_k(e_k)}$ is infinite, let $v_k$ be the common endpoint of $e_k$ and $e'_k$. Note that $e'_k=ge_k$ for some $g\in G_{v_k}$ of infinite order, otherwise the stabilizer of $f(e_k)$ would be equal to $\langle G_{e_k},G_{e'_k}\rangle$ or $\langle G_{e_k},h\rangle$ for some $h$ of finite order, but these groups are finite since $G_e$ is virtually cyclic with infinite center, and this is a contradiction. The subgroup $\langle g\rangle$ has finite index in $G_e$ and fixes the vertex $v_k$. It follows that $G_e$ is elliptic in $T_k$ as well. Let $x$ be the point the closest to $v_k$ that is fixed by $G_e$ in $T_k$. Assume towards a contradiction that $x\neq v_k$. Every edge in the segment $[x,v_k]\subset T_k$ has infinite stabilizer since $g$ fixes $[x,v_k]$. Moreover, the stabilizer of every edge $\varepsilon$ in $[x,v_k]$ is strictly contained in $G_e$, by definition of $x$. In addition, $\phi_k(\varepsilon)\neq e$ by the maximality property for $T_k$. Consequently, $f_k$ is non-injective on $[x,v_k]$. Thus, there exist two adjacent edges in $[x,v_k]$ with infinite stabilizers that are identified by $f_k$. It follows from the maximality property for $T_k$ that these two edges have the same infinite stabilizer, and belong to distinct $G$-orbits. Hence, we could perform a fold of type 1 in the tree $T_k$. This contradicts the priority order, since $f_k$ is not a fold of type 1 (because $G_{e_k}$ is finite). So we have proved that $G_e$ fixes $v_k$. As a consequence, we can replace $f_k: T_k\twoheadrightarrow T_{k+1}$ by the fold identifying $e_k$ with $he_k$ for every $h\in G_e$. This new tree $T_{k+1}$ has the maximality property.
\end{enumerate}

Now, let $(T_k)$ be a sequence of trees, with $T_0=T$, which respects the priority order (type 1 before type 2 before type 3) and the maximality property. Let $T_n$ be the last tree along the folding sequence that contains an edge with finite stabilizer. We claim that $T_{n+1}=T'$. Assume towards a contradiction that $T_{n+1}\neq T'$, and let us prove that we could perform a fold of type 1 in $T_n$, contradicting the order of priority in the sequence of folds.

Let us observe that $G$ acts without inversion on $T_{n+1}$. Indeed, $G$ acts without inversion on $T'$, and $T'$ is obtained from $T_{n+1}$ by a sequence of folds.

Since the fold $f_n$ involves an edge with finite stabilizer, it is not of type 1. If $f_n$ is of type 2, it is defined by a pair of adjacent edges $(e_n=[x,y],e'_n=[x,y'])$ in $T_n$. If $f_n$ is of type 3, it is defined by a pair $(e_n=[x,y], K)$ where $K$ is a subgroup of $G_x$ (see Remark \ref{...}). Up to exchanging $e_n$ and $e'_n$ (in the case where $f_n$ is of type 2), one can assume that the stabilizer of $e_n$ is finite.

All possible folds in $T_{n+1}$ identify two edges $e=[f_n(v),f_n(w)]$ and $e'=[f_n(v),f_n(w')]$ in distinct $G$-orbits such that $H:=G_e=G_{e'}$ is infinite (by the maximality property). Let $\varepsilon$ be an edge in the preimage of $e$ by $f_n$, and let $\varepsilon'$ be an edge in the preimage of $e'$ by $f_n$. 

First, let us prove that the group $H$ fixes an extremity of $\varepsilon$ (and similarly an extremity of $\varepsilon'$). If $G_{\varepsilon}$ is infinite, then the maximality property implies that $G_{\varepsilon}=G_e=H$. If $G_{\varepsilon}$ is finite, then one can assume without loss of generality that $\varepsilon=e_n$.
\begin{itemize}
\item[$\bullet$]If $f_n$ is of type 2, then $H=G_e=\langle G_{e_n},G_{e'_n}\rangle=G_{e'_n}$ by the maximality property. Since $x$ is an endpoint of the edge $e'_n$, the group $H=G_{e'_n}$ fixes $x$.
\item[$\bullet$]If $f_n$ is of type 3, then $H=G_e=\langle G_{e_n},K\rangle=K$ by the maximality property, and $K$ is a subgroup of $G_x$ by definition. Thus, $H$ fixes $x$.
\end{itemize}

Hence, the group $H$ fixes an extremity of both $\varepsilon$ and $\varepsilon'$, denoted respectively by $v$ and $v'$. If $\varepsilon$ and $\varepsilon'$ were disjoint, then $H$ would fix the segment between $v$ and $v'$. Up to replacing $v$ (resp. $v'$) by the other endpoint of $\varepsilon$ (resp. $\varepsilon'$) if necessary, one can suppose that $v$ and $v'$ are sent on the same point by $f_n$ (because $e=f_n(\varepsilon)$ and $e'=f_n(\varepsilon')$ are adjacent in $T_{n+1}$). As a consequence, there are two adjacent edges $a$ and $a'$ in the segment $[v,v']$ such that $f_n(a)=f_n(a')$. Since $H$ is infinite and fixes $a$ and $a'$, these two edges belong to distinct $G$-orbits, and $G_a=G_{a'}$, according to the maximality property. This is a contradiction, because $f_n$ is of type 2 or 3 since it involves an edge with finite stabilizer.

Therefore, the edges $\varepsilon$ and $\varepsilon'$ are adjacent in $T_n$. Moreover, they belong to distinct $G$-orbits since $e$ and $e'$ belong to distinct $G$-orbits. At least one of the edges $\varepsilon$ and $\varepsilon'$, say $\varepsilon$, has finite stabilizer, otherwise we could perform a fold of type 1 in $T_n$ identifying $\varepsilon$ and $\varepsilon'$, and this would contradict the order of priority in the folding sequence. 

\begin{center}
\begin{tikzpicture}[scale=1]
\node[draw,circle, inner sep=1.7pt, fill] (C1) at (-4,0) {};
\node[draw,circle, inner sep=1.7pt, fill] (C2) at (-4,2) {};
\node[draw,circle, inner sep=1.7pt, fill] (C3) at (-6,1) {};
\node[draw,circle, inner sep=1.7pt, fill] (A1) at (2,0) {};
\node[draw,circle, inner sep=1.7pt, fill] (A2) at (2,2) {};
\node[draw,circle, inner sep=1.7pt, fill] (A3) at (0,1) {};
\node[draw=none, label=below:{$\varepsilon$}] (D1) at (-5,0.5) {};
\node[draw=none, label=below:{$\varepsilon'$}] (D2) at (-5,2.2) {};
\node[draw=none, label=below:{$e$}] (B1) at (1,0.5) {};
\node[draw=none, label=below:{$e'$}] (B2) at (1,2.2) {};
\node[draw=none] (B3) at (7,2) {};
\node[draw=none, label=below:{$f_n$}] (Z4) at (-2,1) {};
\node[draw,circle, inner sep=1.7pt, fill] (A4) at (8,1) {};
\node[draw,circle, inner sep=1.7pt, fill] (A5) at (6,1) {};
\draw[->,>=latex, dashed] (-3,1) to (-1,1);
\draw[-,>=latex] (A3) to (A1) ;
\draw[-,>=latex] (A3) to (A2);
\draw[-,>=latex] (A4) to (A5);
\draw[-,>=latex] (C3) to (C1) ;
\draw[-,>=latex] (C3) to (C2);
\draw[->,>=latex, dashed] (3,1) to (5,1);
\end{tikzpicture}
\end{center}

Let us bserve in addition that it is possible to fold $\varepsilon$ and $\varepsilon'$ in $T_n$, and that this fold is of type 2 since $\varepsilon$ and $\varepsilon'$ are not in the same orbit. As a consequence, the fold $f_n$ is necessarily of type 2 (if it were of type 3, we should have folded $\varepsilon$ and $\varepsilon'$ before, according to the priority order). 

Since $G_{\varepsilon}$ is finite, one can assume without loss of generality that $\varepsilon=e_n$. Thus, one has $G_e=\langle G_{e_n},G_{e'_n}\rangle$. It follows from the maximality property that $G_{e'_n}=G_e=H$. The edges $\varepsilon=e_n$ and $\varepsilon'$ being adjacent in $T_n$, the edges $e'_n$ and $\varepsilon'$ are adjacent as well. Moreover, they are identified in $T'$, since $f_n(e'_n)=e$ and $f_n(\varepsilon')=e'$ are identified in $T_{n+1}$ (in particular in $T'$). In addition, note that $e'_n$ and $\varepsilon'$ lie in distinct $G$-orbits since $f_n(e'_n)=e$ and $f_n(\varepsilon')=e'$ lie in distinct $G$-orbits. Thus, we could have performed a fold of type 1 in $T_n$ by identifying $e'_n$ and $\varepsilon'$, a contradiction.\end{proof2}

We can now prove Lemma \ref{cyclic2}. Recall that this lemma claims that if $G$ is a virtually free group with a centered splitting $\Delta$, then $G$ has no non-degenerate $\Delta$-preretraction. Assume towards a contradiction that $G$ has a non-degenerate $\Delta$-preretraction. Then by Lemma \ref{deuxieme} above, $G$ has a subgroup $H$ with the following property: there exists a non-trivial minimal splitting $\Gamma$ of $H$ over virtually cyclic groups with infinite center such that, for every vertex $x$ of $\Gamma$, the vertex group $H_x$ does not split non-trivially over a finite group relative to the stabilizers of edges incident to $x$ in $\Gamma$. But $H$ is virtually free, as a subgroup of the virtually free group $G$, so Lemma \ref{penible} above tells us that there exists a vertex $x$ of $\Gamma$ such that $H_x$ splits non-trivially over a finite group relative to the stabilizers of edges incident to $x$ in $\Gamma$. This is a contradiction.

\color{black}

\section{A non-$\exists$-homogeneous virtually free group}\label{contre-exemple}

In this section, we give an example of a virtually free group which is not $\exists$-homogeneous. By the way, this example shows that the second step in the prove of the homogeneity of $\mathrm{SL}_2(\mathbb{Z})$ fails in general.

More precisely, we shall construct a virtually free group $G=A\ast_C B$, with $A,B$ finite, and two elements $x,y\in G$ such that:
\begin{itemize}
\item[$\bullet$]there exists a monomorphism $G\hookrightarrow G$ which interchanges $x$ and $y$ (in particular, $\mathrm{tp}_{\exists}(x)=\mathrm{tp}_{\exists}(y)$);
\item[$\bullet$]there is not any automorphism of $G$ which sends $x$ to $y$.
\end{itemize}

This is a new phenomenon, that does not occur in free groups, as shown by the following proposition (see \cite{OH11} Lemma 3.7).

\begin{prop}\label{libre}Let $x$ and $y$ be two elements of the free group $\mathbf{F}_n$. The two following statements are equivalent.
\begin{enumerate}
\item There exists a monomorphism $G\hookrightarrow G$ which sends $x$ to $y$, and a monomorphism $G\hookrightarrow G$ which maps $y$ to $x$.
\item There exists an automorphism of $G$ which maps $x$ to $y$.
\end{enumerate}
\end{prop}

Here is a proof of Proposition \ref{libre} above.

\begin{proof}
Let $H_x<\mathbf{F}_n$ be a free factor relative to $x$ and let $H_y<\mathbf{F}_n$ be a free factor relative to $y$. Let $\phi : \mathbf{F}_n \hookrightarrow \mathbf{F}_n$ be a monomorphism $G\hookrightarrow G$ which sends $x$ to $y$, and let $\psi : \mathbf{F}_n \hookrightarrow \mathbf{F}_n$ a monomorphism $G\hookrightarrow G$ which sends $y$ to $x$. One easily sees that $\phi(H_x)\subset H_y$ and $\psi(H_y)\subset H_x$. The monomorphism $(\psi \circ \phi)_{\vert H_x}$ sends $H_x$ into itself and it fixes $x$. As a consequence, it is an automorphism of $H_x$ thanks to the relative co-Hopf property \ref{coHopf}. Hence, $\phi$ induces an isomorphism from $H_x$ to $H_y$. It remains to extend $\phi$ to an automorphism of $\mathbf{F}_n$. Write $\mathbf{F}_n=H_x\ast K_x=H_y\ast K_y$. The groups $K_x$ and $K_y$ are free and have the same rank, so there is an isomorphism $\alpha : K_x\rightarrow K_y$. One define an automorphism $\psi$ of $\mathbf{F}_n$ that sends $x$ to $y$ by $\psi_{\vert H_x}=\phi_{\vert H_x}$ and $\psi_{\vert K_x}=\alpha_{\vert K_x}$.\end{proof}

\begin{rque}It follows from the previous proposition that the free group $\mathbf{F}_2=\langle a,b\rangle$ is $\exists$-homogeneous. Indeed, if $x$ and $y$ have the same existential type, one can express by means of an existential sentence that there exists an endomorphism $\phi$ of $F_2$ that sends $x$ to $y$ and that does not kill the commutator $[a,b]$. The subgroup $\phi(\mathbf{F}_2)<\mathbf{F}_2$ is generated by the two elements $\phi(a)$ and $\phi(b)$ that do not commute, so $\phi(\mathbf{F}_2)$ is isomorphic to $\mathbf{F}_2$ and $\phi$ is injective. Likewise, there exists a monomorphism $\psi : \mathbf{F}_2\hookrightarrow \mathbf{F}_2$ that sends $y$ to $x$. Now, the previous proposition implies that there exists an automorphism of $\mathbf{F}_2$ that maps $x$ to $y$.\end{rque}

The following proposition shows that the elements $x$ and $y$ must have finite order in the counterexample.

\begin{prop}\label{111}Let $G=A\ast_C B$, where $A$ and $B$ are finite. Suppose that $A$ has no subgroup isomorphic to $B$, and that $B$ has no subgroup isomorphic to $A$. Let $x$ and $y$ be two elements of $G$ which have the same $\exists$-type. If $x$ has infinite order, then there exists an automorphism of $G$ that maps $x$ to $y$.\end{prop}

\begin{proof}As in the proof of the homogeneity of $\mathrm{SL}_2(\mathbb{Z})$, one can prove that there exist a monomorphism $\phi$ of $G$ that maps $x$ to $y$, and a monomorphism $\psi$ of $G$ that maps $y$ to $x$. Moreover, one can prove that $G$ is one-ended relative to $x$. Since $x$ has infinite order, the element $y$ has infinite order as well, so $G$ is one-ended relative to $y$. Theorem \ref{coHopf} tells us that $G$ is co-hopfian relative to $x$. Therefore, the monomorphism $\psi\circ \phi$ is an automorphism of $G$, since it fixes $x$. As a consequence, $\phi$ and $\psi$ are two automorphisms.\end{proof}

\subsection{Definition of the group we seek}

In Section \ref{motivation}, we proved that the group $\mathrm{SL}_2(\mathbb{Z})$ is $\exists$-homogeneous. Similarly, one can prove the following proposition that give sufficient conditions under which a group of the form $G=A\ast_C B$, with $A,B$ finite and $C\lhd G$, is $\exists$-homogeneous.

\begin{prop}\label{222}Let $G=A\ast_C B$, where $A$ and $B$ are finite, and $C\lhd G$. Suppose that $A$ has no subgroup isomorphic to $B$, and that $B$ has no subgroup isomorphic to $A$. Assume moreover that, for every $g\in G$, there are two elements $a\in A$ and $b\in B$ such that $\mathrm{ad}(g)_{\vert C}=\mathrm{ad}(ba)_{\vert C}$. Then the group $G$ is $\exists$-homogeneous.\end{prop}

For instance, if $C$ is cyclic, or more generally if $\mathrm{Out}(C)$ is abelian, the group $G$ is $\exists$-homogeneous. 

\begin{proof}Let $x$ and $y$ be two elements of $G$ which have the same $\exists$-type. If $x$ has infinite order, then it follows from Proposition \ref{111} above that there exists an automorphism of $G$ that maps $x$ to $y$. Now, suppose that $x$ has finite order. As in the proof of the homogeneity of $\mathrm{SL}_2(\mathbb{Z})$, one can prove that there exist a monomorphism $\phi$ of $G$ that maps $x$ to $y$ and a monomorphism $\psi$ of $G$ that maps $y$ to $x$. We claim that it is possible to modify $\phi$ in such a way as to obtain an automorphism of $G$ which sends $x$ to $y$. First, observe that there exist two elements $g$ and $h$ of $G$ such that $\phi(A)=A^g$ and $\phi(B)=B^h$. Up to composing $\phi$ by $\mathrm{ad}(h^{-1})$, one may suppose that the element $h$ is trivial. By assumption, there are two elements $a\in A$ and $b\in B$ such that $\mathrm{ad}(bag^{-1})$ coincides with the identity map on $C=\phi(C)$. Thus, one can define an endomorphism $\alpha$ of $G$ by setting $\alpha_{\vert A}=\mathrm{ad}(bag^{-1})\circ \phi$ and $\alpha_{\vert B}=\phi$. This homomorphism is surjective since its image contains $A$ and $B$ (indeed, $\alpha(A)=A^b$ and $\alpha(B)=B$). Hence, $\alpha$ is an automorphism of $G$. Moreover, since $y=\phi(x)$ has finite order, $\alpha$ sends $y$ to a conjugate $y^{\gamma}$ of $y$. Thus, the automorphism $\mathrm{ad}(\gamma^{-1})\circ \alpha$ maps $x$ to $y$.\end{proof}

In order to construct a non-$\exists$-homogeneous group $G$ of the form $G=A\ast_C B$, where $A$ and $B$ are finite, it is necessary to violate the conditions given by Proposition \ref{222} above. Note in addition that the elements $x$ and $y$ we are looking for need have finite order, by Proposition \ref{111}. We are now ready to define the group $G$ together with the elements $x$ and $y$.

Set $C=(\mathbb{Z}/2\mathbb{Z})^4$ and $e_1=(1,0,0,0)$, $e_2=(0,1,0,0)$, $e_3=(0,0,1,0)$, $e_4=(0,0,0,1)$. Let $(e_i \ e_j)$ denote the element of $\mathrm{Aut}(C)=\mathrm{GL}_4(\mathbb{F}_2)$ that interchanges $e_i$ and $e_j$ while leaving fixed $e_k$ for $k\notin \lbrace i,j\rbrace$. Let us define a homomorphism $f$ from $(\mathbb{Z}/2\mathbb{Z})^2=\langle x\rangle\times \langle y\rangle$ to $\mathrm{Aut}(C)$ by setting $f(x)=(e_1 \ e_2)$ and $f(y)=(e_3 \ e_4)$. Let $A=C\rtimes_f (\langle x\rangle\times \langle y\rangle)$. Then, let us define a homomorphism $h$ from $\mathbb{Z}/3\mathbb{Z}=\langle z\rangle$ to $\mathrm{Aut}(C)$ by setting $h(z)= (e_1 \ e_2 \ e_3)$, and let $B=C\rtimes_h \langle z\rangle$. Finally, let $G=A\ast_CB$.

\subsection{A monomorphism that interchanges $x$ and $y$} First, let us consider the following automorphism $\psi$ of $A$:
\begin{center}
$\psi(x)=y$, 

$\psi(y)=x$, 

$\psi_{\vert C}=(e_1 \ e_3)(e_2 \ e_4 )$.
\end{center}
Let $u=z^{-1}xyz\in G$. One easily checks that $\psi$ and $\mathrm{ad}(u)$ coincide on $C$, so one can define an endomorphism $\phi$ of $G$ by setting $\phi_{\vert B}=\mathrm{ad}(u)$ and $\phi_{\vert A}=\psi$. Note that $\phi(C)=C$ since $\psi$ preserves $C$. Since $\phi$ induces an automorphism $\psi$ of $A$, and restricts to a conjugation on $B$, it sends any reduced normal form to a reduced normal form. It follows that $\phi$ is injective.

\subsection{There does not exist any automorphism which sends $x$ to $y$.}

Observe that the subgroup $\langle x,z\rangle<G$ generated by $x$ and $z$ fixes $e_4$. This subgroup acts by conjugation on $C$ as the permutation group $S_3$, whereas the subgroup $\langle y,z\rangle<G$ acts on $C$ as $S_4$. Assume towards a contradiction that there exists an automorphism $\sigma$ of $G$ sending $x$ to $y$. Since the deformation space of the splitting $A\ast_C B$ is reduced to a point and invariant under the automorphisms of $G$, the Bass-Serre tree $T$ is invariant under automorphisms as well. Hence, there exists a $\sigma$-equivariant isometry $f : T \rightarrow T$. Up to composing $\sigma$ by an inner automorphism, since $G$ acts transitively on the set of edges of $T$, one can assume that $f$ preserves the edge $[v,w]$ such that $G_v=A$ and $G_w=B$. Thus, $\sigma(A)=A$, $\sigma(B)=B$ and $\sigma(C)=C$. Consequently, $\sigma(z)=z^{\pm 1}c$ for some $c\in C$. It follows that the action of $\langle\sigma(z),y\rangle$ on $C$ by conjugation is the same as the action of $\langle z,y\rangle$, which acts as the permutation group $S_4$. But $\sigma(\langle z,x\rangle)=\langle\sigma(z),\sigma(x)\rangle=\langle\sigma(z),y\rangle$ acts as $S_3$ on $C$. This is a contradiction.

\section{Uniform almost-homogeneity in virtually free groups}\label{presque-homogénéité}

In the current section, we prove our main result.

\begin{te}Virtually free groups are uniformly almost-homogeneous.\end{te}

Let us recall Corollary \ref{corol}, which will play a key role in the proof of the theorem above.

\begin{co}\label{corol2}
Let $G$ be a virtually free group. Let $u,v\in G^k$. Denote by $\Delta_u$ and $\Delta_v$ two Stallings splittings of $G$ relative to $u$ and $v$ respectively. Let $U$ and $V$ be the vertex groups of $\Delta_u$ and $\Delta_v$ which contain $u$ and $v$ respectively. If $u$ and $v$ have same type, then there exists an endomorphism $\phi$ of $G$ which sends $u$ to $v$ and induces an automorphism between $U$ and $V$.
\end{co}

Let us prove the theorem.

\begin{proof}Let $G$ be a virtually free group, let $k\geq 1$ be an integer and let $u$ be a $k$-tuple of elements of $G$. Let $(u_n)_{n\in\mathbb{N}}$ be a sequence of $k$-tuples of elements of $G$ such that $u_0=u$ and $\mathrm{tp}(u)=\mathrm{tp}(u_n)$ for every $n$. We shall prove that there exist two distinct integers $n\neq m$ and an automorphism $\sigma$ of $G$ such that $\sigma(u_n)=u_m$. More precisely, we shall prove that there exists an integer $N\geq 1$ which does not depend on $u$, $(u_n)$ and $k$, such that $\vert \lbrace u_n\rbrace/\mathrm{Aut}(G)\vert \leq N$.

For every integer $n$, let $T_n$ denote the Bass-Serre tree associated with a reduced Stallings splitting of $G$ relative to $\langle u_n\rangle$, and let $U_n$ denote the vertex group that contains $u_n$. Let $\phi_n$ be an endomorphism of $G$ that sends $u=u_0$ to $u_n$ and that induces an isomorphism from $U=U_0$ to $U_n$ (this endomorphism $\phi_n$ does exist by Corollaire \ref{corol} reminded above).

We shall prove that there are only finitely many $U_n$ modulo $\mathrm{Aut}(G)$. For each $n$, there exists a non-redundant tree $S_n$ in the Stallings deformation space of $G$, together with a collapse $\pi_n : S_n \twoheadrightarrow T_n$. Indeed, let $S_n$ be the tree obtained from $T_n$ by replacing the vertex fixed by $U_n$ by a reduced Stallings splitting of $U_n$. This tree $S_n$ is a Stallings splitting of $G$, and it collapses onto $T_n$. Up to forgetting the vertices of degree 2, one can assume that $S_n$ is non-redundant.

By Proposition \ref{compacité}, the Stallings deformation space of $G$ is cocompact in the following sense: there exist finitely many trees $X_1,\ldots ,X_p$ in the deformation space such that, for every non-redundant tree $S$ in the Stallings deformation space of $G$, $S=X_i^{\sigma}$ for some $1\leq i\leq p$ and some $\sigma\in\mathrm{Aut}(G)$. As a consequence, up to extracting a subsequence from $(u_n)$, one can suppose that there exists a tree $S$ in the deformation space such that for every $n$, there exists an automorphism $\sigma_n\in\mathrm{Aut}(G)$ such that $S_n=S^{\sigma_n}$. But there are only finitely many ways in which we can collapse $S$, since there are only finitely many orbits of edges under the action of $G$. Hence, up to extracting a subsequence from $(u_n)$ once again, one can suppose that there exists a splitting $T$ of $G$ such that $T_n=T^{\sigma_n}$ for every $n$. Therefore, there exist two distinct integers $n$ and $m$ such that $U_n=\sigma_n\circ \sigma_m^{-1}(U_m)$.

Up to extracting a subsequence from $(u_n)$, and up to replacing each $u_n$ by an element belonging to the same orbit under $\mathrm{Aut}(G)$, one can now assume that all the $u_n$ lie in the same relative one-ended factor, for instance $U$. Therefore, for each $n$, the homomorphism $\phi_n$ induces an automorphism $\alpha_n$ of $U$ that sends $u$ to $u_n$. Hence, $\alpha_m\circ \alpha_n^{-1}$ is an automorphism of $U$ that sends $u_n$ to $u_m$. 

Since the group $U$ has finitely many conjugacy classes of finite subgroups, there exist two distinct integers $n\neq m$ such that the restriction of $\alpha_m\circ \alpha_n^{-1}$ to any finite subgroup of $U$ is a conjugation by some element of $G$. As a consequence, $\alpha_m\circ \alpha_n^{-1}$ is a conjugation on each stabilizer of an edge incident to the vertex fixed by $U$ in $T_u$. Thus $\alpha_m\circ \alpha_n^{-1}$ extends to an automorphism $\alpha$ of $G$.

We have proved that $\vert \lbrace u_n\rbrace/\mathrm{Aut}(G)\vert$ is bounded from above by a constant that does not depend on $u$, $(u_n)$ and $k$.\end{proof}

\section{Generic homogeneity in hyperbolic groups}\label{generic}

It is natural to wonder to what extent a non-homogeneous hyperbolic group is far from being homogeneous. In the current section we shall prove that, in a probabilistic sense, the deficiency of homogeneity is negligible. To that end, let us introduce the following definition.

\begin{de}Let $G$ be a group. Let $k$ be an integer $\geq 1$. A $k$-tuple $u$ of elements of $G$ is said to be \textit{type-determined} if $\vert\lbrace v\in G^k \ \vert \ \mathrm{tp}(v)=\mathrm{tp}(u)\rbrace/\mathrm{Aut}(G)\vert =1$.\end{de}

Let $\mu$ be a probability measure on a finitely generated group $G$ whose support is finite and generates $G$ as a semigroup. An element of $G$ arising from a random walk on $G$ of length $n$ generated by $\mu$ is called a \textit{random element} of length $n$. We define a \textit{random $k$-tuple} of length $n$ as a $k$-tuple of random elements of length $n$ arising from $k$ independant random walks. The following result holds.

\begin{te}\label{random}Let $k$ be an integer $\geq 1$. In a hyperbolic group, the probability that a random $k$-tuple of length $n$ is type-determined tends to one as $n$ tends to infinity.\end{te}

\begin{rque}In fact, we shall prove the following stronger result: if $u\in G^k$ is rigid, then it is $\exists$-\textit{type-determined}, meaning that any $k$-tuple $v$ with the same $\exists$-type as $u$ belongs to the same $\mathrm{Aut}(G)$-orbit.\end{rque}

First, we prove Theorem \ref{random} for virtually free groups. In this case, the previous theorem is an easy consequence of Proposition \ref{prop1} combined with a result of Maher and Sisto \cite{MS17}.

\subsection{Generic homogeneity in virtually free groups}

Let $G$ be a virtually free group. Say a tuple $u$ of elements of $G$ is \textit{one-ended} if the group $G$ is one-ended relative to $\langle u\rangle$. The proof of Theorem \ref{random} can be divided into two parts. Fix an integer $k\geq 1$ and let $u\in G^k$.

\emph{Step 1.} If $u$ is one-ended, then $u$ is type-determined (see Proposition \ref{stepg1}). This is an easy consequence of Proposition \ref{prop1}.

\emph{Step 2.} The probability that a random $k$-tuple of length $n$ is one-ended tends to one as $n$ tends to infinity (see Proposition \ref{stepg2}).

\begin{prop}\label{stepg1}Let $G$ be a virtually free group and let $u$ be a finite tuple of elements of $G$. If $u$ is one-ended, then $u$ is type-determined.\end{prop} 

\begin{proof}Let $v$ be a tuple of elements of $G$ such that $\mathrm{tp}_{\forall\exists}(u)=\mathrm{tp}_{\forall\exists}(v)$. Since $G$ is one-ended relative to $\langle u\rangle$, it follows from Proposition \ref{prop1} that there exists a monomorphism $\phi : G \hookrightarrow G$ that sends $u$ to $v$. Let $V$ be the maximal one-ended subgroup of $G$ relative to $\langle v\rangle$. By Proposition \ref{prop1}, there exists an endomorphism $\phi : G \rightarrow G$ that sends $v$ to $u$ and whose restriction to $V$ is injective. Moreover, $\phi(G)$ is contained in $V$. As a consequence, $\psi\circ \phi$ is a monomorphism of $G$ that fixes $u$, so it is an automorphism since $G$ is co-Hopfian relative to $\langle u\rangle$ (see Theorem \ref{coHopf}). Thus, $\phi$ is an automorphism of $G$ which maps $u$ to $v$.\end{proof}

Before proving the second part, we need some preliminary results. Let $G$ be a virtually free group, and let $T$ be a Stallings tree of $G$. We say that two elements $g$ and $g'$ of $G$ have a $p$\textit{-match} if their axes $\gamma$ and $\gamma'$ in $T$ have translates whose intersection has length greater than $p$. We shall use the following result, which is a particular case of Proposition 10 of \cite{MS17}. 

\begin{prop}\label{MS}Let $G$ be a virtually free group, let $T$ be a Stallings tree of $G$. Let $\mu$ be a probability distribution on $G$ such that $\mathrm{supp}(\mu)$ generates $G$ as a semigroup. Let $g\in G$ be an element of infinite order which lies in the support à $\mu$, and let $\gamma$ denote the axis of $g$. We denote by $\gamma_n$ the axis of a random element of length $n$. There exists a constant $C\geq 0$ such that, for every $p\geq C$, the probability that $\gamma$ and $\gamma_n$ have a $p$-match tends to $1$ as $n$ tends to infinity.\end{prop}

We keep the same notations. Let $v$ be a vertex of $T$. Let $A_v$ be the set of vertices of $T$ adjacent to $v$. Let $g\in G$ be an element of infinite order. The \textit{Whitehead graph} $\mathrm{Wh}_T(g,v)$ is the labeled graph defined as follows: its vertex set is $A_v$, and two vertices $v_1$ and $v_2$ are joined by an edge if the axis of a conjugate of $g$ contains the segment $[v_1,v_2]$. Whitehead graphs were first introduced by Whitehead in the context of free groups to give a criterion for characterizing primitive elements. Our presentation is inspired from \cite{GH17}, in which the authors extend Whitehead criterion to free products of groups.

The following proposition gives a sufficient condition on an element $g\in G$ for being one-ended, in terms of Whitehead graphs defined above. We say that $g$ \textit{fills} $T$ if, for every vertex $v\in T$, the graph $\mathrm{Wh}_T(g,v)$ is complete.

\begin{prop}\label{rigid}Let $G$ be a virtually free group, let $T$ be a Stallings tree of $G$, and let $g$ be an element of $G$ of infinite order. If $g$ fills $T$, then $g$ is one-ended.\end{prop}

\begin{proof}
Towards a contradiction, suppose that $G$ has a splitting of the form $A\ast_C B$ or $A\ast_C$ with $C$ finite, such that $g\in A$. Let $T_A$ be the $A$-invariant minimal subtree of $T$. It is a Stallings tree of $A$. Let $T_B$ be a Stallings tree of $B$. Let $S$ be the Stallings tree of $G$ obtained by replacing the vertex fixed by $A$ by $T_A$ in the Bass-Serre tree of $A\ast_C B$, and by replacing the vertex fixed by $B$ by $T_B$. There exists a cellular map $f : S \rightarrow T$ which coincides with the identity map on $T_A$. The axis $\gamma$ of $g$ is contained in $T_A$, thus it avoids $T_B$. Consequently, if $f$ is an isometry, the axis of $f(\gamma)$ avoids $f(T_B)$. This is a contradiction since $g$ fills $T$ by assumption (i.e.\ for every vertex $v$ of the quotient graph $T/G$, the graph $\mathrm{Wh}_T(g,v)$ is complete). Thus, $f$ is not an isometry, so it folds two edges or collapses an edge. 

\emph{First case.} Assume that $f$ maps two adjacent edges $[v,w]$ and $[v,w']$ of $S$ to the same edge $[f(v),f(w)=f(w')]$ in $T$. Let $A_{w}$ (resp.\ $A_{w'}$) let denote the set of vertices of $S$ which are adjacent to $w$ (resp.\ to $w'$). Since the restriction of $f$ to the axis $\gamma$ of $g$ is an isometry, this axis cannot contain the segment $[w,w']$. Likewise, none of the translates $h\cdot \gamma$ of $\gamma$ (with $h\in G$) contain $[w,w']$. As a consequence, none of the translates of $f(\gamma)$ contain a segment $[f(x),f(x')]$ with $x\in A_w$ and $x'\in A_{w'}$. So the Whitehead graph $\mathrm{Wh}_T(g,f(w)=f(w'))$ is not complete. This is a contradiction.

\emph{Second case.} Let $[v,w]$ be an edge of $S$ collapsed by $f$. Let $A_v$ be the set of vertices of $S$ that are adjacent to $v$ and that are not of the form $h\cdot w$ with $h\in G_v$. We define $A_w$ in the same manner. Since $f$ is an isometry on $\gamma$, none of the translates of $\gamma$ contain the edge $[v,w]$, thus none of the translates of $f(\gamma)$ contain the segment $[f(x),f(y)]$ with $x\in A_v$ and $y\in A_w$. So the graph $\mathrm{Wh}_T(g,f(v)=f(w))$ is not complete. This is a contradiction.\end{proof}

We are ready to prove the second step of the proof of Theorem \ref{random} in the case of virtually free groups.

\begin{prop}\label{stepg2}Let $G$ be a virtually free group, let $T$ be a Stallings tree of $G$, and let $g$ be an element of infinite order that fills $T$. Let $\mu$ be a probability distribution on $G$ such that $\mathrm{supp}(\mu)$ generates $G$ as a semigroup and contains $g$. Then, the probability that a random element of length $n$ is one-ended tends to one as $n$ tends to infinity.\end{prop}

\begin{proof}Let $T$ be a Stallings tree of $G$. Let $L$ be the translation length of $g$. By Proposition \ref{MS}, there exists a constant $C\geq 0$ such that, for every $p\geq \max(C,L)$, the probability that $g$ and a random element $g_n$ of length $n$ have a $p$-match tends to $1$ as $n$ tends to infinity. Moreover, if $g$ and $g_n$ have a $p$-match, the element $g_n$ fills $T$ since $p\geq L$. Now, it follows from Proposition \ref{rigid} that $g_n$ is one-ended.\end{proof}

\subsection{Generic homogeneity in the general case}

A tuple $u$ of elements of $G$ is termed \textit{rigid} if the group $G$ does not split non-trivially relative to $\langle u\rangle$. In a work in progress, Guirardel and Levitt prove the following result, which claims that a hyperbolic group has no non-trivial splitting relative to a random tuple, i.e.\ that a random tuple is rigid. Their proof relies on Proposition 10 in \cite{MS17}.

\begin{prop}Fix an integer $k\geq 1$. In a hyperbolic group, the probability that a random $k$-tuple of length $n$ is rigid tends to one as $n$ tends to infinity.\end{prop} 

Theorem \ref{random} for hyperbolic groups is an immediate corollary of the previous result, combined with the following proposition.

\begin{prop}Fix an integer $k\geq 1$. Let $G$ a hyperbolic group, and let $u=(u_1,\ldots,u_k)$ be a $k$-tuple of element of $G$. If $u$ is rigid, then $u$ is type-determined. In fact, the following stronger property holds: any $k$-tuple $v$ with the same $\exists$-type as $u$ belongs to the same $\mathrm{Aut}(G)$-orbit.\end{prop}

\begin{proof}
Since $u$ is rigid, the JSJ splitting of $G$ relative to $\langle u\rangle$ is trivial and the modular group $\mathrm{Mod}_{\langle u\rangle}(G)$ is trivial as well. It follows from Theorem \ref{short} that there exists a finite subset $F=\lbrace w_1,\ldots ,w_{\ell}\rbrace\subset G\setminus \lbrace 1\rbrace$ such that any non-injective endomorphism $\phi$ of $G$ which fixes $u$ kills an element of $F$. We claim that this statement is expressible by means of an existential formula $\theta(\mathbf{y})$ with $k$ free variables satisfied by $u$.

Let $G=\langle s_1,\ldots ,s_n \ \vert \ \Sigma(s_1,\ldots,s_n)\rangle$ be a finite presentation of $G$. Let $v=(v_1,\ldots ,v_k)$ be a $k$-tuple of elements of $G$ such that $\mathrm{tp}_{\exists}(v)=\mathrm{tp}_{\exists}(u)$. Observe that there is a one-to-one correspondance between the set of endomorphisms of $G$ and the set of solutions in $G^n$ of the system of equations $\Sigma(x_1,\ldots,x_n)=1$. We define the formula $\theta(\mathbf{y})$ as follows:

\[\theta(\mathbf{y}) : \exists x_1\ldots \exists x_n \ \Sigma(x_1,\ldots,x_n)=1\bigwedge_{i=1}^k y_i=u_i(x_1,\ldots ,x_n) \bigwedge_{i=1}^{\ell} w_i(x_1,\ldots,x_n)\neq 1 .\]

Since the identity of $G$ fixes $u$ and does not kill any element of $F$, the statement $\theta(u)$ holds in $G$ (by setting $x_i=s_i$ for every $1\leq i\leq n$). Since $u$ and $v$ have the same existential type, $\theta(v)$ is satisfied by $G$ as well. Hence, one can define an endomorphism $\phi$ of $G$ which maps $u$ to $v$ by sending $s_i$ to $x_i$ for every $1\leq i\leq n$. This endomorphism is injective by Theorem \ref{short}, so it is an automorphism by Theorem \ref{coHopf}.\end{proof}

\renewcommand{\refname}{References}
\bibliographystyle{alpha}
\bibliography{biblio}

\def\cprime{$'$} \def\cprime{$'$}
\begin{thebibliography}{And18}

\bibitem[And18]{And18}
Simon Andr\'e.
\newblock Hyperbolicity and cubulability are preserved under elementary
  equivalence.
\newblock arXiv:1801.09411, 2018.

\bibitem[DG11]{DG11}
Fran{\c{c}}ois Dahmani and Vincent Guirardel.
\newblock The isomorphism problem for all hyperbolic groups.
\newblock {\em Geom. Funct. Anal.}, 21(2):223--300, 2011.

\bibitem[GH17]{GH17}
Vincent Guirardel and Camille Horbez.
\newblock Algebraic laminations for free products and arational trees.
\newblock arXiv:1709.05664, 2017.

\bibitem[GL15]{GL15}
Vincent Guirardel and Gilbert Levitt.
\newblock Splittings and automorphisms of relatively hyperbolic groups.
\newblock {\em Groups Geom. Dyn.}, 9(2):599--663, 2015.

\bibitem[GL17]{GL16}
Vincent Guirardel and Gilbert Levitt.
\newblock J{SJ} decompositions of groups.
\newblock {\em Ast\'{e}risque}, (395):vii+165, 2017.

\bibitem[Gui08]{Gui08}
Vincent Guirardel.
\newblock Actions of finitely generated groups on {$\Bbb R$}-trees.
\newblock {\em Ann. Inst. Fourier (Grenoble)}, 58(1):159--211, 2008.

\bibitem[Hor17]{Hor17}
Camille Horbez.
\newblock The boundary of the outer space of a free product.
\newblock {\em Israel J. Math.}, 221(1):179--234, 2017.

\bibitem[Mar02]{Mar02}
David Marker.
\newblock {\em Model theory}, volume 217 of {\em Graduate Texts in
  Mathematics}.
\newblock Springer-Verlag, New York, 2002.
\newblock An introduction.

\bibitem[Moi13]{Moi13}
Christophe Moioli.
\newblock Graphes de groupes et groupes co-hopfiens.
\newblock {\em PhD thesis}, 2013.

\bibitem[MS17]{MS17}
Joseph Maher and Alessandro Sisto.
\newblock Random subgroups of acylindrically hyperbolic groups and hyperbolic
  embeddings.
\newblock arXiv:1701.00253, 2017.

\bibitem[Nie03]{Nie03}
Andr\'e Nies.
\newblock Aspects of free groups.
\newblock {\em J. Algebra}, 263(1):119--125, 2003.

\bibitem[OH11]{OH11}
A.~Ould~Houcine.
\newblock Homogeneity and prime models in torsion-free hyperbolic groups.
\newblock {\em Confluentes Math.}, 3(1):121--155, 2011.

\bibitem[Per08]{Per08}
Chlo{\'e} Perin.
\newblock {\em {Elementary embeddings in torsion-free hyperbolic groups}}.
\newblock Theses, {Universit{\'e} de Caen}, October 2008.
\newblock Th{\`e}se r{\'e}dig{\'e}e en anglais, avec une introduction
  d{\'e}taill{\'e}e en fran{\c c}ais.

\bibitem[Per11]{Per11}
Chlo{\'e} Perin.
\newblock Elementary embeddings in torsion-free hyperbolic groups.
\newblock {\em Ann. Sci. \'Ec. Norm. Sup\'er. (4)}, 44(4):631--681, 2011.

\bibitem[PS12]{PS12}
Chlo\'e Perin and Rizos Sklinos.
\newblock Homogeneity in the free group.
\newblock {\em Duke Math. J.}, 161(13):2635--2668, 2012.

\bibitem[RS94]{RS94}
E.~Rips and Z.~Sela.
\newblock Structure and rigidity in hyperbolic groups. {I}.
\newblock {\em Geom. Funct. Anal.}, 4(3):337--371, 1994.

\bibitem[RW14]{RW14}
Cornelius Reinfeldt and Richard Weidmann.
\newblock Makanin-razborov diagrams for hyperbolic groups.
\newblock 2014.

\bibitem[Sel97]{Sel97}
Z.~Sela.
\newblock Structure and rigidity in ({G}romov) hyperbolic groups and discrete
  groups in rank {$1$} {L}ie groups. {II}.
\newblock {\em Geom. Funct. Anal.}, 7(3):561--593, 1997.

\bibitem[Sta83]{Sta83}
John~R. Stallings.
\newblock Topology of finite graphs.
\newblock {\em Invent. Math.}, 71(3):551--565, 1983.

\end{thebibliography}

\vspace{5mm}

\textbf{Simon André}

Univ Rennes, CNRS, IRMAR - UMR 6625, F-35000 Rennes, France.

E-mail address: \emph{simon.andre@univ-rennes1.fr}

\end{document}